\documentclass{amsart}
\usepackage{amsmath,amssymb,amsthm,amsfonts,amscd,mathrsfs,tikz-cd}
\usepackage{tikz}
\usepackage[all,cmtip]{xy}
\usepackage{graphicx}
\usepackage{subcaption}
\usetikzlibrary{decorations.markings,arrows,automata,arrows,backgrounds,snakes}

\theoremstyle{plain}
\newtheorem*{introtheorem}{Theorem}
\newtheorem{theorem}{Theorem}[section]
\newtheorem{proposition}[theorem]{Proposition}
\newtheorem{lemma}[theorem]{Lemma}
\newtheorem{corollary}[theorem]{Corollary}

\newtheorem*{proposition*}{Proposition}

\theoremstyle{definition}
\newtheorem{definition}[theorem]{Definition}
\newtheorem{example}[theorem]{Example}

\theoremstyle{remark}
\newtheorem{remark}[theorem]{Remark}

\newcommand{\secref}[1]{Section~\ref{#1}}
\newcommand{\thmref}[1]{Theorem~\ref{#1}}
\newcommand{\propref}[1]{Proposition~\ref{#1}}
\newcommand{\lemref}[1]{Lemma~\ref{#1}}
\newcommand{\corref}[1]{Corollary~\ref{#1}}

\newcommand{\exref}[1]{Example~\ref{#1}}

\newcommand{\remref}[1]{Remark~\ref{#1}}

\newcommand{\defref}[1]{Definition~\ref{#1}}
\newcommand{\figref}[1]{Figure~\ref{#1}}

\newcommand{\R}{\mathbb{R}}
\newcommand{\Z}{\mathbb{Z}}

\newcommand{\cl}{\mathrm{cl}}
\newcommand{\e}{\mathrm{e}}
\newcommand{\lk}{\mathrm{lk}}

\begin{document}
\tikzset{->-/.style={decoration={
  markings,
  mark=at position .4 with {{->}}},postaction={decorate}}}

\title[Digital Fundamental Groups]{Digital Fundamental Groups and Edge Groups of Clique Complexes}

\author{Gregory Lupton}
\author{Nicholas A. Scoville}

\address{Department of Mathematics, Cleveland State University, Cleveland OH 44115 U.S.A.}

\email{g.lupton@csuohio.edu}

\address{Department of Mathematics and Computer Science, Ursinus College, Collegeville PA 19426 U.S.A.}

\email{nscoville@ursinus.edu}

\date{\today}


\begin{abstract} In  previous work, we have defined---intrinsically, entirely within the digital setting---a fundamental group for digital images.  Here, we show that this group is isomorphic to the  edge group of the clique complex of the digital image considered as a graph.      The clique complex is a simplicial complex and its edge group is well-known to be isomorphic to the ordinary (topological) fundamental group of its geometric realization.  This  identification of our intrinsic digital  fundamental group with a topological fundamental group---extrinsic to the digital setting---means that many familiar facts about the ordinary fundamental group may be translated into their counterparts for the digital fundamental group:    The digital fundamental group of any digital circle is $\Z$; a version of the Seifert-van Kampen Theorem holds for our digital fundamental group; every finitely presented group occurs as the (digital) fundamental group of some digital image.  We also show that the (digital) fundamental group of every 2D digital image is a free group.
\end{abstract}

\maketitle

\section{Introduction}

A \emph{digital image} $X$ is a finite subset $X \subseteq \Z^n$ of the integral lattice in some $n$-dimensional Euclidean space, together with
a particular adjacency relation on the set of points.
 This is an abstraction of an actual digital image which consists of  pixels (in the plane, or higher dimensional analogues of such).
\emph{Digital topology} refers to the use of notions and methods from (algebraic) topology to study digital images.
The idea in doing so is that such notions can provide useful theoretical background for certain steps of image processing, such as  contour filling, border and boundary following, thinning, and feature extraction or recognition (e.g. see p.273 of \cite{Ko-Ro91}).  There is an extensive literature on digital topology (e.g. \cite{Ro86, Bo99, Evako2006}).
As a contribution to this literature,  in \cite{LOS19c, LOS19a, LOS19b} we have started to  build a  general ``digital homotopy theory" that brings the full strength of homotopy theory to the digital setting.  In \cite{LOS19c}  we focussed on the fundamental group.  Our definition of the digital fundamental group  in  \cite{LOS19c}---see below for a r{\'e}sum{\'e}---is \emph{intrinsic}, in the sense that it is defined directly in terms of a digital image, using ingredients such as homotopy of based loops defined within the digital setting.  Indeed, a crucial component of our development in  \cite{LOS19c} involves the notion of \emph{subdivision} of a digital image---a construction that relies on the ``cubical" setting of the integer lattice and which does not translate out of the digital setting in any obvious way.   One of the main results of \cite{LOS19c} shows  that this process of subdivision preserves the fundamental group of  a digital image (Th.3.16 of  \cite{LOS19c}).

In this paper, we make significant advances on the development of \cite{LOS19c}.  The main result is the following.

\begin{introtheorem}[\thmref{thm: digital pi1 = edge pi1}]
Let $X$ be a digital image and $\cl(X)$ its clique complex.  The digital fundamental group of $X$, as defined in \cite{LOS19c}, is isomorphic to the edge group of $\cl(X)$.
\end{introtheorem}
See below for descriptions of  the clique complex and of the edge group of a simplicial complex.  Now it is known that the edge group of a simplicial complex is isomorphic to the fundamental group---in the ordinary, topological sense---of the spatial realization of the simplicial complex (see \cite[Th.3.3.9]{Mau96}, repeated as \thmref{thm: edge gp = pi1} here).
It follows that the relatively unfamiliar digital fundamental group may be identified with the much more familiar topological fundamental group of a space that is associated to the digital image in a fairly transparent way.  With this identification we may, with care over one or two technical points, translate many known results about the topological fundamental group  into their counterparts for the digital fundamental group.  Doing so adds greatly to our understanding of the digital fundamental group.

An overview of the organization of the paper and our results follows.   \secref{sec: basics} summarizes some basics of digital topology and our definition of the digital fundamental group from \cite{LOS19c}.  We have tried to keep this material to the minimum necessary for understanding our results here, and refer to \cite{LOS19c} for fuller details.
In \secref{sec: based homotopy} we give two technical results about relative homotopy of paths or loops.  These results were not included in  \cite{LOS19c}, so we prove them here  since they are needed in the sequel.  \secref{sec: edge groups} contains our main result.  We review clique complexes and edge groups, and prove the isomorphism asserted in the Theorem above.   In \secref{sec: first consequences}	 we begin to draw consequences from this Theorem.  In \thmref{thm: pi of C is Z} we show that the digital fundamental group of any digital circle is $\Z$ (we define what we mean by digital circles in \defref{def: circle}).  In \thmref{thm: DVK} we deduce a version of the Seifert-van Kampen theorem for the digital fundamental group.  The conclusion is the same as the topological theorem, but we require an extra (mild) hypothesis in addition to the usual connectivity hypotheses.  We use  this result to give concrete examples of digital images with interesting fundamental groups.  \exref{ex: DD} shows that a one-point union of two digital circles has non-abelian digital fundamental group (a free group on two generators, in fact).
\exref{ex: projective plane} shows that a certain digital image---which we construct as a ``digital projective plane"---has torsion in its digital fundamental group (which is $\Z_2$, in fact).
These examples are deduced from special cases of our digital Seifert-van Kampen theorem (\corref{cor: contractible U cap V} and \corref{cor: contractible V}).  To the best of our knowledge, these are the first examples given of digital images with fundamental group---in any sense---that is not free abelian.  More generally, we are able to realize any finitely presented group as the digital fundamental group of some digital image in Theorem \ref{thm: fg realization}.  In the final \secref{sec: 2D free}, we show that the digital fundamental group of every 2D digital image is a free group.  This result does not follow automatically from  the isomorphism of \thmref{thm: digital pi1 = edge pi1}.  Rather, we establish it after some preliminary results in \secref{sec: 2D free} about shortening of paths that are of interest in their own right.

The fundamental group is not new in digital topology (see \cite{Kong89, Bo99}, for example).  But our approach and development in \cite{LOS19c} and here differs from versions previously used in digital topology.  We give some discussion of these differences now.  As we pointed out in \cite{LOS19c}, our fundamental group differs from that of \cite{Bo99} for basic examples of digital images.  This difference derives from differences in the notion of homotopy, and is explained in some detail in \cite{LOS19c}.
Ayala et al.~\cite{ADFQ03} work in a setting in which digital images have extra structure that our notion of digital image does not have \emph{a priori}.   By making different choices of their ``weak lighting function," for example, one can arrive at different notions of a fundamental group that on a digital circle take  $\Z$ or the trivial group.    Furthermore, \cite{ADFQ03}  does not actually define a fundamental group in the digital setting.  Rather, their ``digital" fundamental group is defined \emph{extrinsically} to be the edge group of an auxiliary complex; they do not work in terms of loops and homotopies in the actual digital image itself, as we do.
A digital image in our sense only conforms to one of the general ``device models" considered in \cite{ADFQ03}, namely, the \emph{standard cubical decomposition of Euclidean $n$-space} $\R^n$. Working within that device model, and using $(3^n-1)$-adjacency in $\Z^n$, as we do consistently, we do not know whether it is possible to make a uniform, once and for all, choice of extra structure  for which the corresponding fundamental group of \cite{ADFQ03}  determined by such a choice agrees with our fundamental group.   If not, then our notions of fundamental group are basically different. But even if it were,  it is unlikely that such a matching would extend to any other aspects of our more general digital homotopy theory.  For example, maps of digital images and homotopies of them do not appear to be discussed in the body of work surrounding \cite{ADFQ03}.

We end this introduction by mentioning a more general notion than that of a digital image to which many of our results apply.
A \emph{tolerance space} is a set with a symmetric, reflexive binary relation (which we interpret as an adjacency relation on  the points of the set).  Poston (in \cite{Po71}) referred to the use of notions from (algebraic) topology in a tolerance space setting as \emph{fuzzy geometry}, and used ``fuzzy" terminology throughout.  Sossinsky, however,  makes a sharp distinction between  tolerance spaces and more general ``fuzzy mathematics" (see \S5, `Tolerance is Crisp, Not Fuzzy,' of \cite{So86}).  For a recent, detailed history of tolerance spaces together with further examples of applications of tolerance spaces, see \cite{Pet12}.   Every digital image is a tolerance space.  Conversely, every finite tolerance space may be embedded in some $\Z^n$ as a digital image, preserving the adjacencies (we explain how in \propref{prop: digital graph}  below).  But there may be many ways to ``realize" a given tolerance space as a digital image.  Thus, a digital image may be thought of as a tolerance space together with a particular choice of embedding into some $\Z^n$.  Our focus is on developing homotopy theory in  the context of digital images.  However, many of our results apply just as well to tolerance spaces.  The main difference between the two concepts, from our point of view, concerns subdivision.  Whereas a digital image has canonical subdivisions (that are defined in terms of the ambient $\Z^n$), a tolerance space does not.  One can always embed a tolerance space as a digital image in some $\Z^n$, and then use the subdivisions for that dimension, but there is no canonical choice of such.  Generally speaking, then, results that we prove about a digital image $X$ may be interpreted equally well as results about a general tolerance space $X$, so long as the proofs do not involve subdividing $X$.   Examples of this include the results of \cite{LOS19c} through Theorem 3.15---including the definition of the fundamental group, its independence of the choice of basepoint, and its behaviour with respect to products.  Also, \thmref{thm: digital pi1 = edge pi1} of this paper and its consequences in \secref{sec: first consequences}	apply equally well to tolerance spaces as to digital images (the proofs involve subdivisions of intervals---the domains of paths and loops, but do not involve subdivisions of the digital image/tolerance space). The result of \secref{sec: 2D free}, on the other hand, is specifically about 2D digital images and would only make sense as a statement about tolerance spaces that may be ``realized" as 2D digital images.

\smallskip

\textbf{Acknowledgements.}  Thanks to John Oprea for many helpful comments on this work.  The second-named author was supported by a travel grant from Ursinus College.   Also,  thanks to  Andrea Bianchi for explaining to one of us (Lupton) the procedure for realizing a finite tolerance space as a digital image, which we give as \propref{prop: digital graph} here.

\section{Digital Topology and a Digital Fundamental Group}\label{sec: basics}

We review some notation and terminology from digital topology, and give a brief summary of our definition of the fundamental group from \cite{LOS19c}.
Because we are dealing with  the fundamental group, our basic object of interest is a \emph{based digital image}, and maps and homotopies will preserve basepoints.

\subsection{Adjacency and Continuity}
A \emph{based digital image} $X$ means a finite subset $X \subseteq \Z^n$ of the integral lattice in some $n$-dimensional Euclidean space, together with a choice of a distinguished point $x_0 \in X$ which we refer to as the \emph{basepoint} of $X$, and the following reflexive, symmetric binary relation on $X$ that we refer to as \emph{adjacency}:  two (not necessarily distinct) points $x = (x_1, \ldots, x_n) \in X$ and  $y = (y_1, \ldots, y_n) \in X$  are adjacent if $|x_i-y_i| \leq 1$ for each $i = 1, \dots, n$.
If $x, y \in X \subseteq \Z^n$, we write $x \sim_X y$ to denote that $x$ and $y$ are adjacent.
We usually suppress the basepoint $x_0$ from our notation unless it is useful to emphasize the particular basepoint.  Thus, we will denote a based digital image $(X, x_0)$ simply as $X$, with  the understanding that there is some choice of basepoint $x_0$.

We use the notation $I_N$ or $[0, N]$  for the \emph{digital interval of length} $N$. Namely, $I_N \subseteq \Z$ consists of the integers from $0$ to $N$ (inclusive)  in $\Z$ where consecutive
integers are adjacent. Thus, we have $I_1 = [0, 1] = \{0, 1\}$, $I_2 = [0, 2] = \{0, 1, 2\}$, and so-on.  Occasionally, we may use $I_0$ to denote the singleton point $\{0\} \subseteq \Z$.
We will consistently choose $0 \in I_N$ as the basepoint of an interval.

For based digital images  $X \subseteq \Z^n$ and $Y \subseteq \Z^m$, a  function $f\colon X \to Y$  is \emph{continuous} if $f(x) \sim_Y f(y)$ whenever $x \sim_Xy$, and is \emph{based} if $f(x_0) = y_0$.
By a \emph{based map} of based digital images, we mean a continuous, based function.

\subsection{Paths, Loops and Homotopies}\label{subsec: review homotopy}
Let $(Y, y_0)$ be a based digital image with $Y \subseteq \Z^n$.  For any  $N \geq 1$, a \emph{based path of length $N$ in $Y$} is a based map $\alpha\colon I_N \to Y$ (with $\alpha(0) = y_0$).  Unlike in the topological setting, where any path may be taken with the fixed domain $[0, 1]$, in the digital setting we must allow paths to have different domains. A \emph{based loop of length $N$} in $Y$ is a based path $\gamma\colon I_N \to Y$ that satisfies $\gamma(0) = \gamma(N) = y_0$.

A based digital image $(X, x_0)$ is \emph{connected} if, for any $x \in X$ there is some based path $\alpha\colon I_N \to X$ (for some $N \geq 0$) with $\alpha(N) = x$.

The product of based digital images $(X, x_0)$ with $X\subseteq \Z^m$ and $(Y, y_0)$ with $Y\subseteq \Z^n$ is $\big(X \times Y, (x_0, y_0)\big)$.  Here,  the Cartesian product  $X \times Y \subseteq \Z^{m} \times \Z^{n} \cong \Z^{m+n}$ has the adjacency relation $(x, y) \sim_{X \times Y} (x', y')$ when $x\sim_X x'$  and $y \sim_Y y'$.

Two based maps of based digital images $f, g\colon X \to Y$ are \emph{based homotopic}  if, for some $N\geq 1$, there is a (continuous) based map
$$H \colon X \times I_N \to Y,$$
with $H(x, 0) = f(x)$ and $H(x, N) = g(x)$, and $H(x_0, t) = y_0$ for all $t=0, \ldots, N$.  Then $H$ is a \emph{based homotopy} from $f$ to $g$, and we write $f \approx g$.

We specialize this to the context of based loops as follows.
Based loops $\alpha, \beta \colon I_M \to Y$ (of the same length) are \emph{based homotopic as based loops} if there is a based homotopy $H\colon I_M \times I_N \to Y$ with $H(0, t) = H(M, t) = y_0$ for all $t \in I_N$. We refer to such a homotopy as a based homotopy of based loops and we write $\alpha \approx \beta$, even though the homotopy is more restrictive here than in the general based sense.  The context should make it clear exactly what we intend our homotopies to preserve.

\subsection{Subdivision of Intervals} In our broader digital homotopy theory program, subdivision of digital images plays a prominent role.  However, for the purposes of this paper we do not need the general notion of subdivision of a digital image.  Rather, we only need subdivision for intervals.  We will restrict ourselves to this particular instance of subdivision here, and refer to our other papers for the more general notion---especially \cite{LOS19b} in which we discuss subdivision of maps as well as of general digital images.

For each $k \geq 2$ and each $N \geq 0$, we have a \emph{standard projection} map
$$\rho_k \colon I_{kN + k-1} \to I_N$$
defined by $\rho_k(i) = \lfloor i/k \rfloor$.  Here,  $\lfloor i/k \rfloor$ denotes the integer part of $i/k$, namely the largest integer less than or equal to $i/k$.  Thus $\rho_k$ aggregates the points of $I_{kN + k-1}$ into groups of $k$ consecutive integers, and sends each aggregate to a suitable point of $I_N$.  The integers $\{0, \ldots, k-1\}$ are sent to $0 \in I_N$,  $\{k, \ldots, 2k-1\}$ are sent to $1$, and so-on.  In \cite{LOS19c} and our other papers, we use notation in the style $S(I_N, k)$, and refer to the $k$-fold subdivision of the interval $I_N$, for what here we are simply taking as the interval  $I_{kN + k-1}$.  We have no need of this general notation here, and so do not adopt it.

Now let $(Y, y_0)$ be a based digital image with $Y \subseteq \Z^n$.
If $\gamma\colon I_N \to Y$ is a based loop in $Y$, then for any $k$,
$$\gamma\circ \rho_k\colon I_{kN + k-1} \to I_N \to Y$$
is also a based loop (of length $kN + k-1$), in that we have  $\gamma\circ \rho_k(0) = \gamma(0) = y_0$ and $\gamma\circ \rho_k(kN+k-1) = \gamma(N) = y_0$.

Geometrically speaking, the composition $\gamma\circ \rho_k$ amounts to a reparametrization of the loop $\gamma$.  The image traced out in $Y$  is the same, but we pause at each point of the loop for an interval of length $k-1$.  This device allows us to compare loops of different lengths, and also provides  flexibility in deforming loops by (based) homotopies.

\subsection{Concatenation of Paths and Loops}
Suppose $\alpha\colon I_M \to Y$ and
$\beta\colon I_N \to Y$ are paths---not necessarily based paths---in $Y$ that satisfy $\alpha(M) \sim_Y \beta(0)$.  Their \emph{concatenation} is  the path $\alpha\cdot\beta\colon I_{M+N+1} \to Y$ of length $M+N+1$ in  $Y$ defined by
\begin{equation}\label{eq: concat}
\alpha\cdot\beta(t) = \begin{cases} \alpha(t) & 0 \leq t \leq M\\  \beta(t-(M+1)) & M+1 \leq t \leq M+N+1.\end{cases}
\end{equation}
If $\alpha(M) = \beta(0)$, then our definition means that we pause for a unit interval when attaching the end of $\alpha$ to the start of $\beta$.

Given two based loops $\alpha\colon I_M \to Y$ and
$\beta\colon I_N \to Y$, we  form their \emph{product} by concatenation:
$$\alpha\cdot\beta\colon I_{M+N+1} \to Y$$
is the based loop of length $M + N +1$ defined by \eqref{eq: concat}.
We pause at  the basepoint for a unit interval when attaching the end of $\alpha$ to the start of $\beta$.    This product of based loops is strictly associative, as is easily checked.

 \subsection{Subdivision-Based Homotopy of Based Loops and the Fundamental Group}
 Two based loops $\alpha \colon I_M \to Y$ and $\beta \colon I_N \to Y$ (generally of different lengths) are \emph{subdivision-based  homotopic as based loops} if, for some $k, l$ with $k, l \geq 1$ and $k(M+1) = l(N+1)$, we have
$$\alpha\circ \rho_k\colon I_{kM + k-1} \to I_M \to Y \quad \text{and} \quad \beta\circ \rho_l\colon I_{lN + l-1}\to I_N \to Y$$
based-homotopic as maps $I_{kM + k-1} = I_{lN + l-1} \to Y$, via a based homotopy of based loops; i.e.,  if we have a homotopy $H\colon I_{kM + k-1} \times I_R \to Y$ that satisfies $H(s, 0) = \alpha\circ \rho_k(s)$ and  $H(s, R) = \beta\circ \rho_l(s)$, and also $H(0, t) = H(kM + k-1, t) = y_0$ for all $t \in I_R$.

In \cite{LOS19c} we show that subdivision-based homotopy of based loops is an equivalence relation on the set of all based loops (of all lengths)  in $Y$.
Denote by $[\alpha]$ the (subdivision-based homotopy) equivalence class of based loops represented by a based loop $\alpha\colon I_N \to Y$.  Thus, we have $[\alpha] = [\alpha\circ \rho_k]$ for any standard projection $\rho_k\colon I_{kN + k-1} \to I_N$. More generally, we write $[\alpha] = [\beta]$ whenever $\alpha$ and $\beta$ are subdivision-based homotopic as based loops in $Y$.

For $Y\subseteq \Z^n$ a based digital image, denote the set of subdivision-based homotopy equivalence classes of based loops in $Y$ by $\pi_1(Y; y_0)$.
As we show in \cite{LOS19c}, setting
$[\alpha]\cdot[\beta] = [\alpha\cdot\beta]$
for based loops $\alpha\colon I_M \to Y$ and $\beta\colon I_N \to Y$ gives a well-defined product on the set $\pi_1(Y; y_0)$.   This product is associative, since concatenation of based loops itself is associative.
Now for any path $\gamma\colon I_M \to Y$, let $\overline{\gamma} \colon I_M \to Y$ denote the \emph{reverse path}
$\overline{\gamma}(t) = \gamma(M-t)$.  If $\alpha$ is a based loop in $Y$, then so too is its reverse $\overline{\alpha}$.
For any $N \geq 0$, write $C_N \colon I_N \to Y$ for the constant loop defined by $C_N(t) = y_0$ for $0 \leq t \leq N$. Since $C_0\circ \rho_k = C_{k-1} \colon I_{k-1} \to Y$ for any $k$, it follows that all the constant loops $C_N$ represent the same subdivision-based homotopy equivalence class of based loops, which we denote by $\mathbf{e}  \in \pi_1(Y; y_0)$.
Then $\pi_1(Y; y_0)$ is a group, with $\mathbf{e}$ a two-sided identity element and
$[\overline{\alpha}]$  a two-sided inverse element of $[\alpha]$, for each $[\alpha] \in \pi_1(Y, y_0)$.  See  \cite{LOS19c} for details of all this

\section{Results on relative homotopy}\label{sec: based homotopy}

In this section, paths need not be based.  

\begin{definition}
Let $X$ be a digital image and suppose $\alpha, \beta \colon I_M \to X$ are (not-necessarily based) paths of the same length with the same initial point and the same terminal point, so $\alpha(0) = \beta(0)$ and $\alpha(M) = \beta(M)$.  (These need not be the same point, unless we want to consider $\alpha$ and $\beta$ as loops.)   Then we say that $\alpha$ and $\beta$ are \emph{homotopic relative the endpoints} if there is a homotopy $H\colon I_M \times I_T \to X$, for some $T$, that satisfies $H(s, 0) =  \alpha(s)$ and $H(s, T) =  \beta(s)$ for $s \in I_M$, as well as $H(0, t) =  \alpha(0) = \beta(0)$ and  $H(M, t) =  \alpha(M) = \beta(M)$ for $t \in T$.  That is, the endpoints of the paths remain fixed under the homotopy. We use the same notation $\alpha \approx \beta$ for this special kind of homotopy as for the ordinary notion of homotopy (in which the endpoints need not be fixed).  Once again, the context should make it clear what we intend our homotopies to preserve.
\end{definition}

\begin{lemma}\label{lem: homotopy rel endpts}
Suppose $\alpha \approx \alpha' \colon I_M \to X$ and  $\beta \approx \beta \colon I_N \to X$ are paths  in a digital image $X$ and that the homotopies are relative the endpoints.  Suppose  that we have  $\alpha(N) = \alpha'(N) \sim_X \beta(0) = \beta'(0)$,  so that we may form the concatenations  $\alpha\cdot\beta$ and   $\alpha'\cdot\beta'$.  Then we have a homotopy of paths relative the endpoints
$$\alpha\cdot\beta \approx \alpha'\cdot\beta' \colon I_{M+N +1} \to X.$$
If the concatenations are of based loops, then this is a based homotopy of based loops.
\end{lemma}

\begin{proof}
This is basically the same as the proof of part (a) of Lemma 3.6 of  \cite{LOS19c}.  We reproduce the proof here.
Suppose we have  homotopies relative the endpoints $H\colon I_M \times I_R \to X$ and $G\colon I_N \times I_T \to X$ from $\alpha$ to $\alpha'$ and from $\beta$ to $\beta'$ respectively.  We first, if necessary, adjust one of the intervals $I_R, I_T$  so that both homotopies are of the same length.  Suppose we have $R < T$ (the case in which  $R > T$ is handled similarly, and we omit it).  Then lengthen $H$ into a based homotopy $H' \colon I_M \times I_T \to X$ defined as
$$H' (s, t) = \begin{cases} H(s, t) & 0 \leq t \leq R \\ H(s, R) & R+1 \leq t \leq T.\end{cases}$$
Allowing this to be continuous on  $I_M \times I_T$, it is clearly a homotopy relative the endpoints from $\alpha$ to $\alpha'$.  To confirm continuity,  say we have $(s, t) \sim_{I_M \times I_T} (s', t')$.  Since $t\sim_{I_T} t'$, we must have either $\{t, t'\} \subseteq [0, R]$ or $\{t, t'\} \subseteq [R, T]$.  If $\{ (s, t),  (s', t')\} \subseteq I_M \times I_R$, then continuity of $H$ gives
$H'(s, t) \sim_{X} H'(s', t')$.  If  $\{ (s, t),  (s', t')\} \subseteq I_M \times [R, T]$, then we have $H'(s, t) = H(s, R) \sim_{X} H(s', R) = H'(s', t')$. It follows that this extended $H'$ is continuous.    Now define a homotopy (with $H' = H$ in case the original $R$ and $T$ are equal)  $H'+G\colon I_{M+N+1} \times I_T \to X$ as
$$(H' +G)(s, t) = \begin{cases} H'(s, t) & 0 \leq s \leq M \\ G(s-(M+1), t) & M+1 \leq s \leq M+N+1.\end{cases}$$
Once again, if continuous on  $I_{M+N+1} \times I_T$, this is clearly a homotopy relative  the endpoints  from $\alpha\cdot\beta$ to $\alpha'\cdot\beta'$.  To check the two homotopies assemble together continuously, we observe  that, if $(s, t) \sim_{I_{M+N+1} \times I_T} (s', t')$, then either $\{s, s'\} \subseteq [0, M+1]$ or $\{s, s'\} \subseteq [M+1, M+N]$.  Then proceeding as in the first part, and using $H(M, t) = \alpha(M) = \alpha'(M)$ and $G(0, t) = \beta(0) = \beta'(0)$, so that $H(M, t) \sim_X G(0, t')$ for all $t, t' \in T$, we confirm the continuity of  $(H' +G)$.
\end{proof}

In our fundamental group, any reparametrization  of the form $\alpha\circ \rho_k$, for a based loop $\alpha$, represents the same equivalence class of loops as $\alpha$ in  $\pi_1(X; x_0)$.  In \cite{Bo99}, a more general kind of reparametrization  of loops was used to form the equivalence classes.  We define this more general reparametrization of paths or loops here.

\begin{definition}\label{def: triv extn}
Let $\alpha \colon I_M \to X$ be a path.   A  \emph{trivial extension} of $\alpha$ is any path $\alpha' \colon I_{M'} \to X$ of the following form.  For each  $i$ with $0 \leq i \leq M$, choose $t_i \in \Z$ with $t_i \geq 0$.  Then define $\alpha'$ by
$$\alpha'(s) = \begin{cases}  \alpha(0) & 0 \leq s \leq t_0\\
\alpha(1) & t_0 + 1 \leq s \leq t_0 + 1 + t_1  \\
\alpha(2) & t_0 + t_1 + 2 \leq s \leq t_0 + t_1 + 2 + t_2  \\
\ \ \vdots & \ \ \vdots \\
\alpha(M) & \sum_{i=0}^{M-1} t_i + M \leq s \leq \sum_{i=0}^{M} t_i + M. \end{cases}
$$
\end{definition}

If we choose each $t_i = 0$, then we retrieve the original path $\alpha$.  Generally, a trivial extension of $\alpha$ is a prolonged version (a re-parametrization) of $\alpha$ that repeats the value $\alpha(i)$ an extra $t_i$ times, to produce a path $\alpha' \colon I_{M'} \to X$ with the same image in $X$ as that of $\alpha$, but of  length
$$M' = \sum_{i=0}^{M} t_i + M.$$
We may also view this trivial extension $\alpha'$ as a concatenation of $M+1$ constant paths
$$\alpha' = \alpha_0 \cdot\ \cdots \cdot \alpha_M$$
where each $\alpha_i \colon I_{t_i} \to  X$ is a constant path of length $t_i$ at $\alpha(i)$.

\begin{lemma}\label{lem: triv ext alpha C}
Let $\alpha \colon I_M \to X$ be any path in $X$.  Suppose we have a trivial extension $\alpha' \colon I_{M'} \to X$ of $\alpha$ as above, with at least one of the $t_i$ positive.  There is a homotopy relative the endpoints
$$\alpha' \approx \alpha \cdot C_T \colon I_{M'} \to X,$$
where $C_T \colon I_T \to X$ is the constant path at $\alpha(M)$ of length $T = \sum_{i=0}^{M} t_i  - 1$.
\end{lemma}

\begin{proof}
Begin with the special case in which one of the $t_i = 1$ and the others are 0 (so we repeat once a single point of $\alpha$).  For each $i \in I_M$, write $\beta_i \colon I_{M+1} \to X$ for the trivial extension of this \emph{elementary} kind defined by
$$\beta_i (s) = \begin{cases} \alpha(s) & 0 \leq s \leq i\\  \alpha(s-1) & i+1 \leq s \leq M+1 .\end{cases}$$

\emph{Claim.} We claim that, for any $M\geq 0$ and each $i$ with $0 \leq i \leq M$, we have a homotopy of paths relative the endpoints $\alpha\cdot C_0 \approx \beta_i \colon I_{M+1} \to X$.

\emph{Proof of Claim.}
Notice that $C_0 \colon I_0 \to X$ is the constant path of length $0$ that maps the singleton point $\{0\}$ to $\alpha(M)$.  Thus we have an equality of paths $\alpha\cdot C_0 = \beta_M \colon I_{M+1} \to X$ for any $M \geq 0$.   Furthermore, we may define a homotopy $H \colon I_{M+1} \times I_1 \to X$ by
$$H(s, t) = \begin{cases} \beta_M(s) & t = 0\\ \beta_{M-1}(s) & t = 1,\end{cases}$$
for any $M \geq 1$.  We check that $H$ is continuous.  For this, suppose we have $(s, t) \sim (s', t')$ in  $I_{M+1} \times I_1$.  If $t = t'$, then we have $H(s, t) \sim_X H(s', t)$ from the continuity of either $\beta_M$ (if $t' = t = 0$) or  $\beta_{M-1}$ (if $t' = t = 1$).  So it remains to check that we have $H(s, 0) \sim_X H(s', 1)$ when $s \sim s'$ in $I_{M+1}$.  Because $s \sim s'$, we must have $\{ s, s'\} \subseteq [0, M-1]$ or $\{ s, s'\} \subseteq [M-1, M+1]$.  If $\{ s, s'\} \subseteq [0, M-1]$, then $H(s, 0) = \beta_M(s) = \alpha(s)$ and $H(s', 1) = \beta_{M-1}(s') = \alpha(s')$.  In this case, then, we have $H(s, 0) \sim_X H(s', 1)$ from the continuity of $\alpha$.  For the remaining choices of $\{ s, s'\} \subseteq [M-1, M+1]$, the possible values for $H$ satisfy $\{ H(s, 0), H(s', 1) \} \subseteq \{ \alpha(M-1), \alpha(M) \}$.  Continuity of $\alpha$ gives that  $\alpha(M-1) \sim_X \alpha(M)$, and it follows that any two values of $H$, when restricted to $[M-1, M+1]\times I_1$, must be adjacent.   Thus $H(s, t) \sim_X H(s', t')$ for any pair of adjacent points; $H$ is a (continuous) homotopy.  Clearly, we have $H(0, t) = \alpha(0)$ and $H(M+1, t) = \alpha(M)$ for $t \in I_1$, and so $H$ is a homotopy relative the endpoints
\begin{equation}\label{eq: elementary te}
\alpha \cdot C_0 = \beta_M \approx \beta_{M-1} \approx I_{M+1} \to X.
\end{equation}
Now assume inductively that, for any $M \geq 0$, we have a  homotopy of paths relative the endpoints $\alpha\cdot C_0 \approx \beta_{M-k}\colon I_{M+1} \to X$, for some $k$ with $0 \leq k \leq M-1$.  Induction starts with $k =0$ or $1$, by the observations we just made leading up to \eqref{eq: elementary te}.  For the inductive step, re-write  $\beta_{M-(k+1)}\colon I_{M+1} \to X$ as a concatenation $\gamma_{M-(k+1)} \cdot \gamma'$ with
$$\gamma_{M-(k+1)}(s) = \beta_{M-(k+1)}(s) \text{ for }  0 \leq s \leq M-k$$
the path of length $M-k$ that agrees with $\beta_{M-(k+1)}$ through the repeated value $\beta_{M-(k+1)}\big(M-(k+1)\big)= \beta_{M-(k+1)}(M-k)$, and
$$\gamma'(s) = \beta_{M-(k+1)}(s+M-k+1)  \text{ for }  0 \leq s \leq k$$
the path of length $k$ that completes  $\beta_{M-(k+1)}$ when concatenated with $\gamma_{M-(k+1)}$.  Then $\gamma_{M-(k+1)}$ is of the form of an elementary trivial extension (but of a path of length $M-k-1$) which we may write as  $\gamma_{M-(k+1)} = \gamma\cdot C'_0$, where $\gamma(s) = \gamma_{M-(k+1)}(s) =  \beta_{M-(k+1)}(s)$ for $0 \leq s \leq M-(k+1)$ and $C'_0$ the constant path of length $0$ at $\gamma\big(M-(k+1)\big) = \gamma_{M-(k+1)}\big(M-(k+1)\big) =  \beta_{M-(k+1)}\big(M-(k+1)\big)$.  As above, define a homotopy $H \colon I_{M-k} \times I_1 \to X$ by
$$G(s, t) = \begin{cases} \gamma_{M-(k+1)}(s) & t = 0\\ \gamma_{M-(k+2)}(s) & t = 1,\end{cases}$$
where $\gamma_{M-(k+2)}$ denotes the elementary trivial extension of $\gamma$ that repeats  the value $\gamma\big(M-(k+2)\big)$.  That is,
$$\gamma_{M-(k+2)}(s) = \begin{cases} \gamma(s) & 0 \leq s \leq M-(k+2)\\ \gamma(s-1) & M-(k+1) \leq s \leq M-k.\end{cases}$$
Exactly as we did leading up to \eqref{eq: elementary te}, we may confirm the continuity of this $G$, and check  that it is a homotopy relative the endpoints
$$\gamma_{M-(k+1)} \approx \gamma_{M-(k+2)}\colon I_{M-k} \to X.$$
From \lemref{lem: homotopy rel endpts} it follows that we have a homotopy relative the endpoints
$$\gamma_{M-(k+1)}\cdot \gamma' \approx \gamma_{M-(k+2)} \cdot \gamma' \colon I_{M-k} \to X.$$
But above, we chose $\gamma_{M-(k+1)}$ so that $\beta_{M-(k+1)} = \gamma_{M-(k+1)} \cdot \gamma'$, and it is easy to see that we have $\beta_{M-(k+2)} = \gamma_{M-(k+2)} \cdot \gamma'$.  Hence we have a homotopy relatiive the endpoints
$$\beta_{M-(k+1)}  \approx \beta_{M-(k+2)}  \colon I_{M+1} \to X,$$
and the induction step is complete.  The claim follows.
\emph{End of Proof of Claim.}

Now a typical trivial extension may be obtained by repeatedly making elementary extensions.  Suppose inductively that the assertion of the lemma  is true for all trivial extensions of $\alpha$ with $T = \sum_{i=0}^{M} t_i   \leq k$, for some $k \geq 1$.  Induction starts with $k=1$, and we have just established this in the claim.  Now say we have a trivial extension $\alpha' \colon I_{M'} \to X$  of $\alpha$ with $M' = M + \sum_{i=0}^{M} t_i $ and $\sum_{i=0}^{M} t_i = k+1$.  Suppose $n$ is the first index for which $t_n >0$.  Then $\alpha'$ is an elementary extension of the path $\alpha'' \colon I_{M'-1} \to X$ defined by
$$\alpha''(s) = \begin{cases} \alpha'(s) & 0 \leq s \leq  \sum_{i=0}^{n} t_i  - 1 \\
\alpha'(s+1) & \sum_{i=0}^{n} t_i \leq s \leq \sum_{i=0}^{M} t_i ,\end{cases}$$
with $M' -1 =k$.  Since $\alpha''$ is a trivial extension of $\alpha$ with $M' -1 = M+T-1$ where $T-1=k$, we may apply the inductive hypothesis to obtain a homotopy relative the endpoints
$$\alpha'' \approx \alpha \cdot C_{T-1}\colon I_{M'-1} \to X.$$
And, because $\alpha'$ is an elementary extension of $\alpha''$ we also  have (from the claim)  a homotopy relative the endpoints
$$\alpha' \approx \alpha'' \cdot C_{0}\colon I_{M'} \to X.$$
Now \lemref{lem: homotopy rel endpts} gives a homotopy relative the endpoints
$$\alpha''  \cdot C_{0} \approx \alpha \cdot C_{T-1}\cdot C_{0} \colon I_{M'} \to X.$$
Transitivity of homotopy relative the endpoints, with the observation that \\ $C_{T-1}\cdot C_{0} = C_T \colon I_T \to X$, now completes the induction.  The result follows.
\end{proof}

\section{Edge Groups and Clique Complexes}\label{sec: edge groups}

The \emph{edge group of a simplicial complex} is a group defined, like the digital fundamental group of a digital image or the fundamental group of a topological space, in terms of equivalence classes of edge loops (namely, loops consisting of edge paths).  The equivalence relation is given by a combinatorial notion of homotopy.  We repeat some of the definitions from \cite[\S3.3]{Mau96}.

Suppose that $K$ is a simplicial complex with $1$-skeleton consisting of vertices $V$ and edges $E$.  An \emph{edge path} is a finite sequence
$\{ v_0, v_1, \ldots, v_n\}$ of vertices in $V$ such that, for each $i$ with $0 \leq i \leq n-1$, we have $v_i = v_{i+1}$ or $\{v_i, v_{i+1} \} \in E$ an edge of $K$.  An edge path is an \emph{edge loop} if we have, in addition, $v_0 = v_n$.

\begin{definition}\label{def: elem edge htpy}
By an \emph{elementary edge-homotopy (relative  the endpoints)} we mean one of the following operations on edge paths:
\begin{itemize}
\item[(a)] If $v_i = v_{i+1}$, for some $i$ with $0 \leq i \leq n-1$, then replace an edge path $\{ v_0,  \ldots, v_i, v_{i+1}, v_{i+2}, \ldots v_n\}$ with $\{ v_0,  \ldots, v_i, v_{i+2}, \ldots v_n\}$.  Namely, delete a repeated vertex.  Or, conversely, for any $i$ with $0 \leq i \leq n$, replace  an edge path $\{ v_0,  \ldots, v_i, v_{i+1}, \ldots v_n\}$ with $\{ v_0,  \ldots, v_i, v_i, v_{i+1}, \ldots v_n\}$.  Namely, insert  a repeat of a vertex.
\item[(b)] If $\{v_{i-1}, v_i, v_{i+1}\}$ form a simplex of $K$,  for some $i$ with $1 \leq i \leq n-1$, replace an edge path $\{ v_0,  \ldots, v_{i-1}, v_i, v_{i+1}, \ldots v_n\}$ with $\{ v_0,  \ldots, v_{i-1}, v_{i+1}, \ldots v_n\}$.   Or, conversely, for any $i$ with $0 \leq i \leq n-1$, replace  an edge path $\{ v_0,  \ldots, v_i, v_{i+1}, \ldots v_n\}$ with $\{ v_0,  \ldots, v_i, v, v_{i+1}, \ldots v_n\}$ for any $v \in V$ for which $\{v_{i}, v, v_{i+1}\}$ form a simplex of $K$.
\end{itemize}
We say that two edge paths are \emph{edge-homotopic (relative their endpoints)} if  one can apply a finite sequence of elementary edge homotopies, of types (a) and (b) in any order or combination, so as to start with one of the edge paths and arrive at the other.  We refer to the sequence of elementary edge homotopies as an \emph{edge homotopy} from one edge path to the other.   If two edge paths $\alpha = \{ v_0, v_1, \ldots, v_n\}$ and $\beta = \{ w_0, w_1, \ldots, w_m\}$ with $v_0 = w_0$ and  $v_n = w_m$ are edge-homotopic (relative their endpoints), then we write $\alpha \approx_\e \beta$.
\end{definition}

If $K$ is a based simplicial complex with basepoint $v_0 \in V$, and if the two edge paths in question are edge loops, each of which starts and finishes at $v_0$, then we will refer to \emph{an edge homotopy of based loops}.  Two edge paths $\alpha = \{ v_0, v_1, \ldots, v_n\}$ and $\beta = \{ w_0, w_1, \ldots, w_m\}$ with $v_n = w_0$ may be concatenated to form  the edge path
$$\alpha \cdot \beta =  \{ v_0, v_1, \ldots, v_n,  w_1, w_2, \ldots, w_m\}.$$

\begin{remark}\label{rem: concatenation}
This concatenation differs from the way in which we concatenate suitable digital paths.  In fact, we could just as well concatenate edge paths in the way in which we do our digital paths, requiring only that $\{ v_n, w_0\}$ be an edge in $K$. However, since we want to cite results from  the literature, we use the standard way of concatenating edge paths.  Doing so causes no problems for us in the development.  
\end{remark}

Suppose that  $K$ is a based simplicial complex with basepoint $v_0 \in V$.  Edge homotopy of based loops is an equivalence relation on the set of edge loops based at $v_0$.  Denote the equivalence class of an edge loop $\alpha$ by $[\alpha]$, and the set of all equivalence classes by $\mathrm{E}(K; v_0)$. Just as for the fundamental group, defining
$$[\alpha] \cdot [\beta] = [\alpha\cdot \beta] \in  \mathrm{E}(K; v_0)$$
gives a well-defined product of equivalence classes.   Concatenation of edge loops is associative, and so this product is associative.  
Each edge path $\alpha = \{ v_0, v_1, \ldots, v_n\}$ has a \emph{reverse}, which is the edge path $\overline{\alpha} = \{ v_n, v_{n-1}, \ldots, v_0\}$.
One confirms that $\overline{\alpha} \cdot \alpha$ and $\alpha\cdot \overline{\alpha}$ are both edge-homotopic, as based loops, to a constant loop at $v_0$.

\begin{remark}
Although intuitively we may think of type (b) elementary edge homotopies as collapsing or expanding a $2$-simplex,  in fact there is no requirement that  $\{v_{i-1}, v_i, v_{i+1}\}$ be a $2$-simplex. Indeed, to reduce  a concatenation of the form  $\alpha\cdot \overline{\alpha}$ to the trivial loop $\{ v_0\}$  requires collapsing terms such as $\ldots, v_i, v_{i+1}, v_i, \ldots$, in which $\{ v_i, v_{i+1}\}$ is an edge, to  $\ldots, v_i,  v_i, \ldots$. 
\end{remark}

Then the equivalence class of the trivial loop $[\{ v_0\}]$ plays the role of a two-sided identity element, and  $[\alpha]^{-1} = [\overline{\alpha}]$ defines inverses, making $\mathrm{E}(K; v_0)$ into a group, called \emph{the edge group} of $K$ (based at $v_0$).

We have the following result.

\begin{theorem}[{\cite[Th.3.3.9]{Mau96}}]\label{thm: edge gp = pi1}
Suppose $K$ is a simplicial complex with basepoint $v_0$.  Let $|K|$ be the spatial realization of $K$, with basepoint $v_0 \in |K|$.  There is an isomorphism of groups
$$ \mathrm{E}(K; v_0) \cong \pi_1(|K|; v_0),$$
where  the right-hand side denotes  the ordinary fundamental group of $|K|$ as a topological space.  \qed
\end{theorem}

This result is often given as a means of computing  the fundamental group of a topological space or, at least, arriving at a presentation of it.  We refer to \cite{Mau96} for details of the material we have just reviewed.

Now suppose that $X$ is a digital image.  We may associate to $X$ its \emph{clique complex}, which we denote by $\cl(X)$ and which is a simplicial complex whose simplices are determined by the cliques of $X$.  Namely, the vertices of $\cl(X)$ are the vertices of $X$.  The $2$-cliques of $X$, namely pairs of adjacent points, are the $1$-simplices of $X$, and so-on.  In general, the $(n+1)$-cliques of $X$ are the $n$-simplices of $\cl(X)$.  Now observe that the set of simplices $\cl(X)$ satisfies the the requirements to be an (abstract) simplicial complex (a subset of a clique is again a clique).

We will show that the (digital) fundamental group of a digital image $X$ is isomorphic to the edge group of the clique complex $\cl(X)$.  The basic idea is to associate to each based loop $\alpha \colon I_M \to X$, in an obvious way, its corresponding edge loop
$$\e(\alpha) = \{ \alpha(0), \alpha(1), \ldots, \alpha(M)\}$$
of vertices in $\cl(X)$.  Note that (digital) continuity of $\alpha$ ensures that $\e(\alpha)$ is an edge path in $\cl(X)$.  Furthermore, it is an edge loop because we also have  $\alpha(0) = \alpha(M) = x_0$. Then we wish to define a homomorphism
$$\phi\colon \pi_1(X; x_0) \to  \mathrm{E}(\cl(X); x_0),$$
by setting $\phi( [\alpha] )= [ \e(\alpha)]$.  From the next lemma, it will follow that this $\phi$ is well-defined; we will complete the proof that $\phi$ gives an isomorphism of groups following that.

\begin{lemma}\label{lem: htpy vs edge htpy}
Let $\alpha, \beta\colon I_M \to X$ and $\gamma\colon I_N \to X$ be based loops in a  digital image $X$.
\begin{itemize}
\item[(i)] For any $k$, we have an edge homotopy of based edge loops $\e(\alpha\circ \rho_k) \approx_\e \e(\alpha)$ in $\cl(X)$.
\item[(ii)] If $\alpha \approx \beta\colon I_M \to X$ as based loops in $X$, then we have an edge homotopy of based edge loops $\e(\alpha) \approx_\e \e(\beta)$ in $\cl(X)$.
\item[(iii)]  If $\alpha$ and $\gamma$ are based-subdivision homotopic as based loops in $X$, then we have an edge homotopy of based edge loops $\e(\alpha) \approx_\e \e(\gamma)$ in $\cl(X)$.
\end{itemize}
\end{lemma}

\begin{proof}
(i)  More generally, if we have any trivial extension $\alpha' \colon I_{M'} \to X$ of $\alpha$, then $\e(\alpha') \approx_\e \e(\alpha)$.  The composition  $\alpha\circ \rho_k$ is simply the special case of a trivial extension of $\alpha$ in which we repeat each value of $\alpha$ a total of $k$-times.  Refer to  \defref{def: triv extn} for our notation about trivial extensions.  Also, in the proof of \lemref{lem: triv ext alpha C}, we defined the \emph{elementary trivial extensions}
$$\beta_i (s) = \begin{cases} \alpha(s) & 0 \leq s \leq i\\  \alpha(s-1) & i+1 \leq s \leq M+1 .\end{cases}$$
These are the special cases of trivial extensions of $\alpha$ in which we repeat once a single point of $\alpha$.  It is tautological that we have
$$\e(\alpha) \approx_\e \e(\beta_i),$$
for each $i$ with $0 \leq i \leq M$, using elementary edge homotopies of type (a) from \defref{def: elem edge htpy} .   Since the composition  $\alpha\circ \rho_k$ may be achieved as a finite sequence  of elementary trivial extensions of the path $\alpha$, so too the edge loop $\e(\alpha\circ \rho_k)$ may be achieved as the corresponding finite sequence  of elementary edge homotopies of type (a) of the edge loop $\e(\alpha)$.

(ii) Suppose we have a based homotopy of based loops $H \colon I_M \times I_N \to X$ from $\alpha$ to $\beta$.  We resolve each step of this homotopy, namely the restriction of $H$ to a map $I _M \times [t, t+1] \to X$ for each $t$ with $0 \leq t \leq N-1$, into a succession of ``elementary homotopies," as follows.

For each $t$ with $0 \leq t \leq N$, define a loop $H_t\colon I _M  \to X$ by $H_t(s) = H(s, t)$ for $s \in I_M$.  Continuity of the homotopy $H$ means that, for each $t$ with $0 \leq t \leq N-1$, the paths $H_t \colon I_M \to X$ and $H_{t+1} \colon I_M \to X$ are adjacent as paths in $X$, in the sense used in \cite{LOS19a}.  Namely, for each $s \sim s'$ in $I_M$, we have $H_t(s) \sim_X H_{t+1}(s')$.  Now define,  for each $q$ with $0 \leq q \leq N-1$, a homotopy
$$G_q \colon I_M \times I_M \to X$$
by setting
$$G_q(s, t) = \begin{cases}  H_{q+1}(s) & 0 \leq s \leq t\\  H_{q}(s) & t+1 \leq s \leq M.\end{cases}$$
If $(s, t) \sim (s', t')$ in $I_M \times I_M$, then in particular we have $s \sim s'$ in $I_M$.  Now the only possible values for $G_q(s, t)$ are $H_{q+1}(s)$ or $H_{q}(s)$, and the only possible values for $G_q(s', t')$ are $H_{q+1}(s') $ or $H_{q}(s')$.  From the remark above, about adjacency of $H_q$ and $H_{q+1}$, it follows that we have both $H_{q+1}(s)$ and $H_{q}(s)$ adjacent to both $H_{q+1}(s')$ and $H_{q}(s')$.  Hence, we have $G_q(s, t) \sim G_q(s', t')$ in $X$, and so $G_q$ is continuous. Notice, then, that \emph{$G_q$ is a homotopy from $H_q$ to $H_{q+1}$} for each $q$ with $0 \leq q \leq N-1$, and that we have $G_q(s, M) = G_{q+1}(s, 0) = H_{q+1}(s)$ for each $s \in I_M$, each $q$ with $0 \leq q \leq N-2$.

Now we may assemble the $G_t$ together into a homotopy
$$G \colon I_M \times I_{MN} \to X$$
by setting
$$G(s, t) = G_q(s, r) \text{ if  } t = Mq+r \text{ for } 0 \leq q \leq N-1 \text{ and } 0 \leq r \leq M-1$$
and then $G(s, MN) = G_{N-1}(s, M) = H_N(s)$.  Note that this $G$ is continuous by the same argument that we use to show homotopy is transitive.  Specifically, here, we have $G = G_q$ when restricted to the rectangle $I_M \times [Mq, M(q+1)]$, and so $G$ is continuous when restricted to each such rectangle.  But if we have $(s, t) \sim (s', t')$ in $I_M \times I_{MN}$, then in particular we have $|t' - t| \leq 1$ and hence both $(s, t)$ and $(s', t')$ must lie in at least one such rectangle.  Then $G(s, t) = G_q(s, t) \sim_X G_q(s', t') = G(s, t)$ for some $q$, and it follows that $G$ is continuous on $I_M \times I_{MN}$.

We have constructed $G$, which is a ``slower" homotopy of based loops from the loop $\alpha$ at which the original homotopy $H$ starts, to the loop $\beta$  at which the original homotopy $H$ ends.  The difference between the two homotopies is that, whereas $H$ makes the transition in unit time from one loop to an adjacent loop that may differ in many values,  the slower homotopy $G$ makes a transition in unit time from one loop to an adjacent loop that differs in at most one value.

\emph{Claim.}  For each $t$ with $0 \leq t \leq MN-1$, define the two based loops $\eta, \eta'\colon I_M \to X$ by $\eta(s) = G(s, t)$ and $\eta'(s) = G(s, t+1)$.  Then we have an edge homotopy of based edge loops $\e(\eta) \approx_\e \e(\eta')$ in $\cl(X)$.

\emph{Proof of Claim.} The loops $\eta$ and $\eta'$ differ in at most one value, which  means that, for some $S$ with $1 \leq S \leq M-1$, we have $\eta(s) = \eta'(s)$ for $s \not= S$, and $\eta(S) \sim_X  \eta'(S)$ (these may agree also, in which case we have $\eta = \eta'$).  This follows from the way in which we have constructed the homotopy $G$.  Then an edge homotopy from $\e(\eta)$ to $\e(\eta')$ is given by the sequence of elementary edge homotopies
$$\begin{aligned} \cdots, \eta(S-1), \eta(S), \eta(S+1), \cdots &\approx_\e \cdots, \eta(S-1), \eta'(S), \eta(S),\eta(S+1), \cdots \\
&\approx_\e \cdots, \eta(S-1), \eta'(S), \eta(S+1), \cdots,
\end{aligned}$$
in which the first edge path is $\e(\eta)$ and the last is  $\e(\eta')$.  The first elementary edge homotopy inserts the vertex $\eta'(S)$ between $\eta(S-1)$ and $\eta(S)$.  This is permissible since $G(S-1, t)$, $G(S, t)$ and $G(S, t+1)$ form a $3$-clique in $X$, from continuity of $G$.  The second elementary edge homotopy deletes the vertex $\eta(S)$ from between $\eta'(S)$ and $\eta(S+1)$.   Again, this is permissible since continuity of $G$ implies that $G(S, t+1)$, $G(S, t)$ and $G(S+1, t)$ form a $3$-clique in $X$.  \emph{End of Proof of Claim.}

Since $\e(\eta) \approx_\e \e(\eta')$ in $\cl(X)$ for each $t$, transitivity of edge homotopies now gives $\e(\alpha) \approx_\e \e(\beta)$ in $\cl(X)$

(iii) Now suppose that $\alpha$ and $\gamma$ are based-subdivision homotopic as based loops in $X$.  This means that for some $k$ and $k'$, we have a based homotopy of based loops $\alpha\circ \rho_k \approx \gamma\circ \rho_{k'}$.  But then we have edge homotopies of based edge loops
$$\e(\alpha) \approx_\e e(\alpha\circ \rho_k) \approx_\e \e(\gamma\circ \rho_{k'}) \approx_\e\e(\gamma)$$
in $\cl(X)$, with the first and third edge homotopies coming from part (i), and the middle edge homotopy from part (ii).  Then part (iii) follows from transitivity of edge homotopies.
\end{proof}

\begin{theorem}\label{thm: digital pi1 = edge pi1}
Let $X$ be a digital image and $\cl(X)$ its clique complex.  The map
$$\phi\colon \pi_1(X; x_0) \to  \mathrm{E}(\cl(X); x_0),$$
defined by setting $\phi( [\alpha] )= [ \e(\alpha)]$ is an isomorphism of groups.
\end{theorem}

\begin{proof}
The map $\phi$ is well-defined by \lemref{lem: htpy vs edge htpy} (recall that, in our formulation of the digital fundamental group $\pi_1(X; x_0)$, based loops $\alpha$ and $\beta$ represent the same element of $\pi_1(X; x_0)$ if  they are subdivision-based homotopic).  Although we concatenate based loops in a slightly different way from that in which edge loops are concatenated (cf.~\remref{rem: concatenation}), nonetheless $\phi$ is a homomorphism.  The concatenation of two based loops $\alpha\cdot \beta$ has a repeat of the basepoint at times $M$ and $M+1$ (if $\alpha$ is of length $M$). But then we may use an elementary edge homotopy of type (a) to delete this repetition, so that we have
$$\e(\alpha\cdot \beta) \approx_\e \e(\alpha) \cdot \e(\beta),$$
where the right-hand side refers to (the standard) concatenation of edge loops in $\cl(X)$. It follows  that $\phi$ is indeed a homomorphism.  Any edge loop $\{v_0, \ldots, v_n\}$ in $\cl(X)$ may be viewed as $\e(\alpha)$, where $\alpha\colon I_n \to X$ is  the path $\alpha(i) = v_i$ for $0 \leq i \leq n$.  Continuity of $\alpha$ follows because $v_i$ and $v_{i+1}$ must be adjacent in $X$ for there to be an edge joining them in $\cl(X)$.  So $\phi$ is evidently onto.

It remains to show that $\phi$ is also injective.  For this it is sufficient to show  that if two edge loops, which---as we just observed---we may assume are of the form $\e(\alpha)$ and $\e(\beta)$ for loops $\alpha$ and $\beta$ in $X$, are homotopic \emph{via} an elementary edge homotopy, then the loops $\alpha$ and $\beta$ are subdivision-based homotopic.  So first suppose that  $\e(\beta)$ is edge homotopic to $\e(\alpha)$ by an elementary edge homotopy of type (a)---addition of a vertex $v_j$ after an occurrence of this vertex in  the edge loop (by the symmetric nature of edge homotopy, it is not necessary to consider removal of a vertex).  Then $\beta$ is what we earlier called an elementary trivial extension of $\alpha$. \lemref{lem: triv ext alpha C} now gives $[\beta] = [\alpha\cdot C_0]$ in $\pi_1(X; x_0)$, with $C_0$ denoting the constant loop at $x_0$.  From \cite{LOS19c} we have $[\alpha\cdot C_0] = [\alpha]\cdot [C_0] = [\alpha]$ in $\pi_1(X; x_0)$ ($\beta$ is subdivision-based homotopic to $\alpha$).  On the other hand, suppose that   $\e(\beta)$ is edge homotopic to $\e(\alpha)$ by an elementary edge homotopy of type (b)---addition of a vertex $v$ between two vertices $\alpha(j)$ and $\alpha(j+1)$ with $\{ \alpha(j), v, \alpha(j+1)\}$ a simplex of $\cl(X)$.      But if $\{ v_j, v, v_{j+1}\}$ is a simplex of $\cl(X)$, then we have $\alpha(j)  \sim \alpha(j+1)$ and $v$ is adjacent to both of these in $X$. Let $\beta_j$ denote the elementary trivial extension of $\alpha$ obtained by repeating the value $\alpha(j)$, as in the proof of \lemref{lem: triv ext alpha C}.  We may define a homotopy
$$H \colon I_{M+1} \times I_1 \to X,$$
assuming $\alpha$ is of length $M$, by setting
$$H(s, t) = \begin{cases} \alpha(s) & 0 \leq s \leq j\\
\alpha(j) & (s, t) = (j+1, 0)\\
v & (s, t) = (j+1, 1)\\
 \alpha(s-1) & j+2 \leq s \leq M+1.\end{cases}$$
 It is easy to confirm that $H$ is continuous, and that it is a based homotopy of based loops $\beta_j \approx \beta$. In $\pi_1(X; x_0)$, then, we have
 $$[\alpha] = [\alpha]\cdot [C_0] = [\alpha\cdot C_0] = [\beta_j] = [\beta],$$
 where the first two re-writes are basic identities in $\pi_1(X; x_0)$, the next is the first item we proved in the proof of \lemref{lem: triv ext alpha C}, and the last follows from the homotopy $H$ above.   Thus, for each type of elementary edge homotopy, we have established that $\e(\alpha) \approx \e(\beta)$ implies $[\alpha] = [\beta] \in \pi_1(X; x_0)$.  Injectivity of $\phi$ follows, and this completes the proof.
\end{proof}

\section{Direct Consequences for  the Digital Fundamental Group}\label{sec: first consequences}

Because so much is known about edge groups of simplicial complexes and the fundamental groups of topological spaces, it is now easy to compile many basic results about the digital fundamental group.  We simply translate known facts and results from the topological setting to the digital setting, wherever feasible.  We begin by considering digital circles.

Our definition of a digital circle  is effectively the same as the ``simple closed curve" definition of \cite[\S3]{Bo99}.   Actually, these curves are closed but are simple only in the  tolerance space sense.

\begin{definition}\label{def: circle}
Consider a set $C = \{x_0, x_1, \ldots, x_{N-1}\}$ of $N$ (distinct) points in $\Z^n$, with $N \geq 4$ and for any $n \geq 2$.  We say that $C$ is a \emph{circle of length $N$}  if we have adjacencies  $x_i \sim_C x_{i+1}$ for each $0 \leq i \leq N-2$, and $x_{N-1} \sim_C x_0$, and no other adjacencies amongst the elements of $C$.
\end{definition}

We may parametrize a digital circle as a loop $\alpha\colon I_N \to X$ (in various ways).

\begin{theorem}\label{thm: pi of C is Z}
$\pi_1(C; x_0) \cong \Z$ for every digital circle $C$.
\end{theorem}

\begin{proof}
The clique complex of a digital circle is a cycle graph, with geometric realization an actual circle $S^1$.  The result follows from \thmref{thm: digital pi1 = edge pi1}, \thmref{thm: edge gp = pi1}, and the well-known, basic calculation of $\pi_1(S^1; x_0) \cong \Z$ (e.g. \cite[Th.II.5.1]{Mas91}).
\end{proof}

\begin{remark}
We have shown in  \cite{LOS19c} that a particular $4$-point digital circle $D$, which we called the diamond, has fundamental group $\pi_1(D; x_0) \cong \Z$.  This computation was done staying within digital topology, using some results we developed in \cite{LOS19a}.  This gives a computation of $\pi_1(S^1; x_0) \cong \Z$ independently of  the usual topological argument, through the identifications
$$\pi_1(S^1; x_0) \cong \pi_1(| \cl(D)| ; x_0) \cong \pi_1(D; x_0)$$
of \thmref{thm: digital pi1 = edge pi1} and  \thmref{thm: edge gp = pi1}.  Furthermore, these theorems allow us to lever the single computation  $\pi_1(D; x_0) \cong \Z$ into a computation of the fundamental group of any digital circle $C$, because we have  $| \cl(D)| = S^1 = | \cl(C)|$ (we mean the spatial realizations are homeomorphic to the circle, here).
Note  that digital circles of different lengths are not (digitally)  based-homotopy equivalent.  We suspect that any two digital circles are subdivision-based homotopy equivalent.  However, we are as yet unable to establish this because the arguments become bogged down in lengthy expositional details.  The digital fundamental group is preserved by this notion of subdivision-based homotopy equivalence.  But the isomorphism  $\pi_1(D; x_0) \cong \pi_1(C; x_0)$, for any digital circle $C$, is available to us without having to establish $D$ and $C$ as subdivision-based homotopy equivalent.   These comments indicate that, speaking generally, enlarged or reduced versions of a digital image should have the same fundamental group as the original, even though they will not be homotopy equivalent, and even though we may not be able to show them subdivision-based homotopy equivalent.  This is because we may---at the fundament group level---pass into the topological setting, enlarge or reduce there, and then pass back into the digital setting.
\end{remark}

We now deduce a general result that enables calculation of many examples.
The Seifert-van Kampen theorem describes the fundamental group of a union $\pi_1(U \cup V; x_0)$ in terms of the fundamental groups $\pi_1(U; x_0)$, $\pi_1(V; x_0)$ and $\pi_1(U \cap V; x_0)$.   We will need to place certain mild constraints on the union.

\begin{definition}\label{def: disconnected complements}
Suppose  $U$ and $V$ are digital images in some $\Z^n$.  Denote by $U' = \{ v \in V \mid v \not\in V \cap U\}$ the complement of $U$ in $U \cup V$ and by $V' = \{ u \in U \mid u \not\in U \cap V\}$ the complement of $V$ in $U \cup V$.  We say that $U$ and $V$ have \emph{disconnected complements} (in $U \cup V$) if  $U'$ and $V'$ are disconnected from each other.  That is,  $U$ and $V$ have disconnected complements when the set of pairs $\{ u, v\}$ with $u \in V'$, $v \in U'$ and $u \sim_{U \cup V} v$ is empty.
\end{definition}

 \begin{theorem}[Digital Seifert-van Kampen]\label{thm: DVK}
Let $U$ and $V$ be connected digital images in some $\Z^n$ with connected intersection $U \cap V$.  Choose $x_0 \in U \cap V$ for the basepoint of $U \cap V$, $U$, $V$, and $U \cup V$.  If $U$ and $V$ have disconnected complements, then
$$\xymatrix{ \pi_1(U\cap V; x_0) \ar[r]^-{i_1} \ar[d]_{i_2} & \pi_1(U; x_0) \ar[d]^{\psi_1}\\
\pi_1(V; x_0) \ar[r]_-{\psi_2} & \pi_1(U\cup V; x_0) }$$
is a pushout diagram of groups and homomorphisms, with $i_1$, $i_2$, $\psi_1$ and $\psi_2$ the homomorphisms of fundamental groups induced by the inclusions $U \cap V \to U$, $U \cap V \to V$, $U \to U \cup V$ and $V \to U \cup V$ respectively.

That is,  suppose we are given any homomorphisms $h_1\colon \pi_1(U; x_0) \to G$ and $h_2\colon \pi_1(V; y_0) \to G$ that satisfy $h_1\circ i_1= h_2\circ i_2 \colon  \pi_1(U\cap V; x_0) \to G$, with $G$ an arbitrary group.  Then there is a homomorphism $\phi\colon \pi_1(U \cup V; x_0) \to G$ that makes (all parts of)  the following diagram commute
\begin{displaymath}
\xymatrix{  \pi_1(U\cap V; x_0) \ar[r]^-{i_1} \ar[d]_{i_2}&  \pi_1(U; x_0) \ar[d]_{\psi_1}
\ar@/^/[ddr]^{h_1}&\\ \pi_1(V; x_0) \ar[r]^-{\psi_2} \ar@/_/[drr]_{h_2}& \pi_1(U\cup V; x_0)
\ar@{.>}[dr]_(.3){\phi}&\\ &&G }
\end{displaymath}
 and $\phi$ is the unique such homomorphism.
 \end{theorem}

\begin{proof}
In general, we have $\cl(U) \cap \cl(V) = \cl(U \cap V)$.  Observe that, with the hypothesis of disconnected complements, we also have $\cl(U) \cup \cl(V) = \cl(U \cup V)$.  Hence, we have  isomorphisms $\pi_1(U \cap V; x_0) \cong E\left( \cl(U) \cap \cl(V); x_0 \right)$ and $\pi_1(U \cup V; x_0) \cong E\left( \cl(U) \cup \cl(V); x_0 \right)$, from \thmref{thm: digital pi1 = edge pi1}.  Now we may apply the  
 ordinary Seifert-van Kampen theorem from the topological setting in the form for simplicial complexes  (see, e.g. \cite[Th.11.60]{Rot95}) 
to the inclusions of connected simplicial (sub-) complexes
$$\xymatrix{ \cl(U) \cap \cl(V) \ar[r] \ar[d] & \cl(U) \ar[d]\\
\cl(V) \ar[r] & \cl(U)\cup \cl(V) }$$
and conclude the result via \thmref{thm: digital pi1 = edge pi1} and \thmref{thm: edge gp = pi1}.
\end{proof}

 \begin{remark}
 We make no assumptions about any of the induced homomorphisms $i_1$, $i_2$, $\psi_1$ and $\psi_2$ being injective.  Depending on the circumstances, some or all of them, in various combinations, may be injective.  But none of them need be injective.
 \end{remark}

 \begin{remark}
 The theorem identifies $\pi_1(U\cup V; x_0)$ up to isomorphism, although it does so indirectly  in terms of a universal property.   For $U$ and $V$ that satisfy the hypotheses, a more concrete description of $\pi_1(U \cup V; x_0)$ may be given as follows (see Th.11.58 of \cite{Rot95} for example). We have an isomorphism
 $$\pi_1(U \cup V; x_0) \cong \frac{ \pi_1(U; x_0) \ast \pi_1(V; x_0)}{N},$$
 where $\pi_1(U; x_0) \ast \pi_1(V; x_0)$ denotes the free product and $N$ the normal subgroup generated by $\{ i_1(g) i_2(g^{-1}) \mid g \in \pi_1(U\cap V; x_0)\}$.   Or, in terms of presentations, if $\pi_1(U; x_0) = \langle G_1 \mid R_1 \rangle$ and $\pi_1(V; x_0) = \langle G_2 \mid R_2 \rangle$, where the $G_i$ and $R_i$ are sets of generators and relations, then we have a presentation
 $$\pi_1(U \cup V; x_0) =  \langle G_1 \cup G_2 \mid R_1 \cup R_2 \cup \{ i_1(g) i_2(g^{-1}) \mid g \in \pi_1(U\cap V; x_0)\} \rangle.$$
 \end{remark}

 \begin{remark}
 The conclusion of the theorem need not hold if $U$ and $V$ do not have disconnected complements.  For example, take $U, V \subseteq \Z^2$ as follows.
 $$U = \{ (1, 0), (0, 1) \}, \quad V = \{ (1, 0), (0, -1), (-1, 0) \},$$
 so that $U \cup V = D$, the diamond, and $U \cap V = \{ (1, 0) \}$.  Then we have $(-1, 0) \sim (0, 1)$ with $(-1, 0) \in U'$ and $(0, 1) \in V'$, and so $U$ and $V$ do not have disconnected complements.  Furthermore, we know from \cite{LOS19c} or \thmref{thm: pi of C is Z} above that $\pi_1\big(D; (1,0)\big) \cong \Z$, whereas here we have $U$ and $V$ both contractible with trivial fundamental group. Evidently, the conclusion of the theorem does not hold.  Specifically, here, the issue is that---concomitant with $U'$ and $V'$ not being disconnected---we have $\cl(U) \cup \cl(V)$ strictly contained in (not equal to)  $\cl(U \cup V)$.
 \end{remark}

\begin{remark}
It is possible to prove \thmref{thm: DVK} entirely within the digital setting (without relying on \thmref{thm: digital pi1 = edge pi1} and \thmref{thm: edge gp = pi1}).  Surprisingly, perhaps, we are able to prove \thmref{thm: DVK} by adapting the argument that is used in \cite{Mas77} to prove the topological Seifert-van Kampen theorem there.  That argument uses the Lebesgue covering lemma, from the theory of compact metric spaces.  In our digital setting, we find that it is possible to follow the same argument without really having to develop a substitute for this ingredient.  It turns out that dividing a rectangle $I_M \times I_N$ into unit squares achieves the same purpose as does dividing the rectangle $I \times I$ into subrectangles of diameter less than the Lebesgue number of a certain covering of $I \times I$ in the topological setting.
\end{remark}

\begin{remark}
Ayala et al.~\cite{ADFQ03} have a Seifert-van Kampen theorem for the digital fundamental groups they consider.  However, as we mentioned in the introduction, their approach is effectively to \emph{define} the fundamental group as that of an associated simplicial complex, so it \emph{a priori} will obey the Seifert-van Kampen theorem and possess any  other properties of  the topological fundamental group.  The difference between that approach and ours is that we have an intrinsic,  self-contained construction of the fundamental group in the digital setting, and we need to establish \thmref{thm: digital pi1 = edge pi1} and  \thmref{thm: edge gp = pi1} in order to make use of the properties of the topological fundamental group.
\end{remark}

There are some special cases of \thmref{thm: DVK} that are especially useful.  First, consider  the case in which the intersection has trivial fundamental group (cf.~\cite[Th.IV.3.1]{Mas91}).

\begin{corollary}[To \thmref{thm: DVK}]\label{cor: contractible U cap V}
Suppose $U$ and $V$ satisfy the hypotheses of \thmref{thm: DVK} (including disconnected complements) and, in addition, we have $\pi_1(U \cap V; x_0) = \{ \mathbf{e} \}$. Then we have
 $$\pi_1(U \cup V; x_0) \cong \pi_1(U; x_0) \ast \pi_1(V;x_0),$$
 where the right-hand side denotes the free product of groups.  More formally,
$$\xymatrix{ \{ \mathbf{e}\} \ar[r]^-{i_1} \ar[d]_{i_2} & \pi_1(U; x_0) \ar[d]^{\psi_1}\\
\pi_1(V; x_0) \ar[r]_-{\psi_2} & \pi_1(U\cup V; x_0) }$$
is a pushout diagram of groups and homomorphisms.
That is,  suppose we are given any homomorphisms $h_1\colon \pi_1(U; x_0) \to G$ and $h_2\colon \pi_1(V; x_0) \to G$ with $G$ an arbitrary group.  Then there is a homomorphism $\phi\colon \pi_1(U \cup V; x_0) \to G$ that makes the following diagram commute
 $$\xymatrix{ \pi_1(U; x_0) \ar[d]_{\psi_1} \ar[rrd]^-{h_1} \\
 \pi_1(U \cup V; x_0) \ar[rr]^{\phi} & &G\\
  \pi_1(V; x_0) \ar[u]^{\psi_2} \ar[rru]_-{h_2}}$$
 and $\phi$ is the unique such homomorphism.
\end{corollary}

\begin{proof}
Direct from \thmref{thm: DVK}.
\end{proof}

In particular, if we have $U \cap V = \{ x_0\}$, so that $U \cup V$ is a one-point union of $U$ and $V$, and if $U$ and $V$ have disconnected complements in $U \cup V$, then we have
$\pi_1(U \cup V; x_0) \cong \pi_1(U; x_0) \ast \pi_1(V;x_0)$.

\begin{example}\label{ex: DD}
Let $D=\{(1,0), (0,1), (-1,0), (0,-1)\}$ be the diamond in $\Z^2$, with basepoint $(1, 0)$.  The ``double diamond" in $\Z^2$ , with basepoint $(0, 0)$ pictured in \figref{fig:D v D} may be viewed as a one-point union $D \vee D$ of two isomorphic copies of $D$.  With $U$ and $V$ the right-hand and the left-hand copies of $D$, respectively, we have $U \cap V = \{(0, 0)$\}, a single point.  Since $\pi_1(D; x_0) \cong \Z$, it follows from \corref{cor: contractible U cap V}  that we have $\pi_1(D \vee D; x_0) \cong \Z \ast \Z$.  Alternatively, we could just as well deduce the same conclusion by observing that $\cl(D \vee D)$ has geometric realization homeomorphic to $S^1 \vee S^1$, the one-point union of two circles, and using the well-known result  that  $\pi_1(S^1 \vee S^1; x_0) \cong \Z \ast \Z$ (e.g., \cite[Ex.IV.3.1]{Mas91}) together with  \thmref{thm: digital pi1 = edge pi1} and \thmref{thm: edge gp = pi1}. This example illustrates that a digital image may have non-abelian fundamental group.
\begin{figure}[h!]
\centering
\includegraphics[trim=140 450 80 100,clip,width=0.5\textwidth]{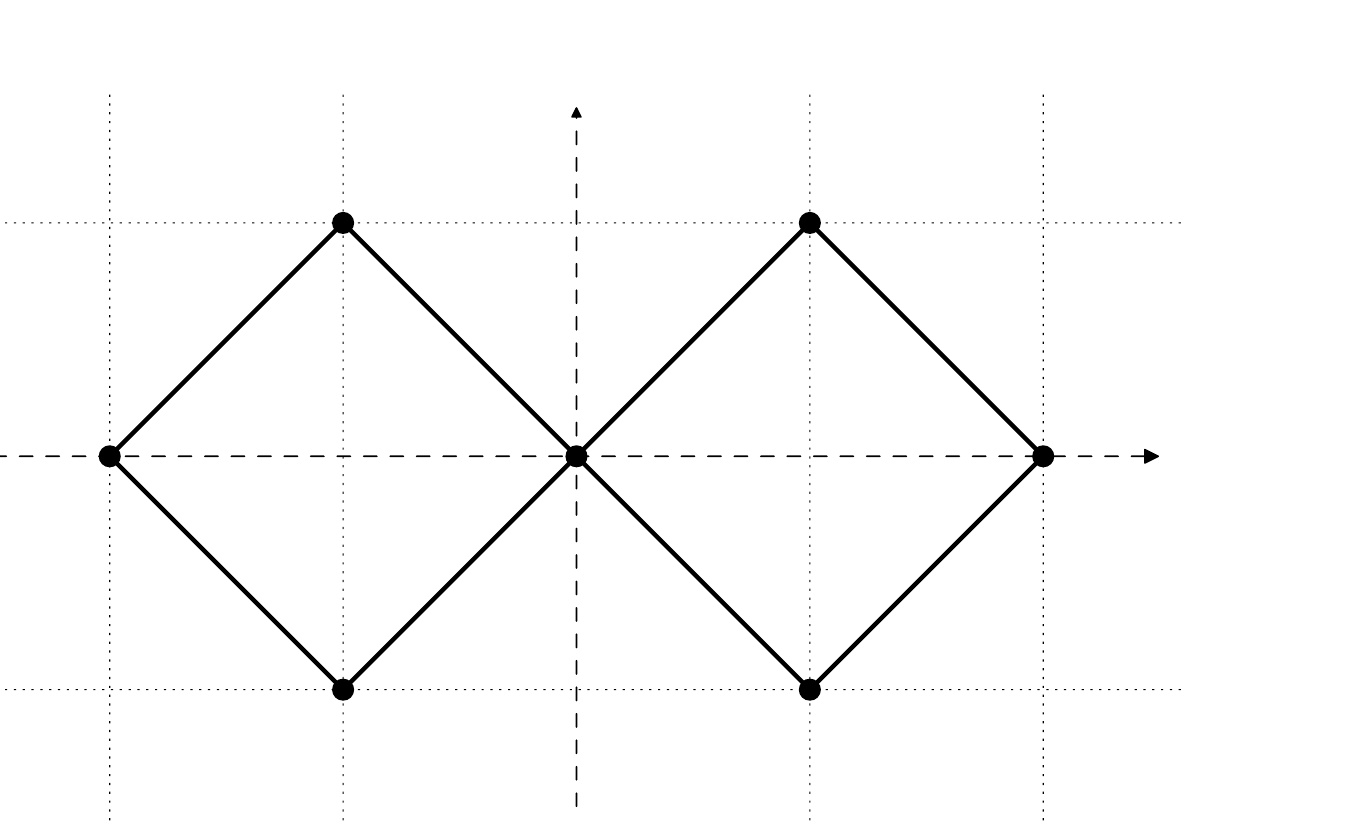}
\caption{$D \vee D$ in $\Z^2$}\label{fig:D v D}
\end{figure}
\end{example}

Another special case  of \thmref{thm: DVK} that is often useful is the case in which one of $U$ or $V$ is contractible or, at least, has trivial fundamental group (cf.~\cite[Th.IV.4.1]{Mas91}.

\begin{corollary}[To \thmref{thm: DVK}]\label{cor: contractible V}
Suppose $U$ and $V$ satisfy the hypotheses of \thmref{thm: DVK} (including disconnected complements) and, in addition, we have $\pi_1(V; x_0)=\{e\}$.   Then $\psi_1\colon \pi_1(U)\to \pi_1(U\cup V)$ is an epimorphism, and its kernel is the smallest normal subgroup of $\pi_1(U)$ containing the image $\phi_1[\pi_1(U\cap V)]$.
\end{corollary}

\begin{proof}
Direct from \thmref{thm: DVK}.
\end{proof}

Our next example will  display a digital image with fundamental group isomorphic to $\Z_2$.  Our approach here is to ``reverse-engineer" a digital image $X$ so that the geometric realization of $\cl(X)$ is homeomorphic to the real projective plane $\R P^2$.    The  approach depends in part on being able to realize a  graph as a digital image.   We now describe a general procedure for doing this.  

Recall our discussion of tolerance spaces from the introduction.  A \emph{simple graph} is one that has no double edges or edges that connect a vertex to itself.  A tolerance space may be viewed as a simple graph, and vice versa, by interpreting ``adjacent vertices" in the tolerance space as ``vertices connected by an edge" in the graph.  In the following, and in the sequel, by an ``isomorphism" across the structures of digital images, on the one hand, and simple graphs/tolerance spaces, on the other, we mean an adjacency-preserving bijection of the vertices with an adjacency-preserving inverse.   

\begin{proposition}\label{prop: digital graph}
If $G$ is a finite simple graph (a finite tolerance space),  then $G$ may be isomorphically embedded as a digital image with vertices in the hypercube $[-1, 1]^{n-1} \subseteq \Z^{n-1}$, where $n = |G|$, the number of vertices. 
\end{proposition}

\begin{proof}
Work by induction on $n$.  Induction starts with $n = 1$ (or $n = 2$), where there is nothing to show.

Inductively assume that, if $|G| \leq n$, then we may embed $G$ as a digital image in  $[-1, 1]^{n-1}$.  Suppose we have a graph $G'$ with $n+1$ vertices.  Choose any vertex $x \in G'$ and write $G' = G \cup \{x\}$ with $|G| = n$.  Embed $G$ as a digital image in  $[-1, 1]^{n-1} \subseteq \Z^{n-1} \subseteq \Z^{n-1} \times \Z = \Z^{n}$. Then each vertex $y \in G$ has coordinates $y = (y_1, \ldots, y_{n-1}, 0) \in \Z^{n}$, and we have $y_i \in \{\pm1, 0\}$ for $i = 1, \ldots, n-1$.  
Denote by $\lk(v) $ the (vertices of the) \emph{link} of a vertex $v$ in a graph, namely, the set of vertices (other than $v$) connected by an edge to $v$.
Now separate the vertices of $G$ into the disjoint union $G = \lk(x) \sqcup \lk(x)^C$. For each $y \in \lk(x)^C$, move it down to the plane $y_n = -1$.  In other words, adjust the embedding of $G$ in $\Z^n$ using the isomorphism of digital images $\phi\colon G \to \overline{G}$ given by
$$\phi(y_1, \ldots, y_{n-1}, 0) = \begin{cases} (y_1, \ldots, y_{n-1}, 0) & \text{if } y \in \lk(x) \\  (y_1, \ldots, y_{n-1}, -1) & \text{if }  y \in \lk(x)^C \end{cases}$$
This is an isomorphism, since we have---for $y, y' \in \Z^{n-1} \times \{0\} \subseteq \Z^n$--- 
$$y \sim_{\Z^n} y' \iff (y_1, \ldots, y_{n-1}) \sim_{\Z^{n-1}} (y'_1, \ldots, y'_{n-1}) \iff \phi(y) \sim_{\Z^n} \phi(y').$$
So we now have $G$ embedded in $\Z^n$ as a digital image with $\lk(x) \subseteq  [-1, 1]^{n-1} \times \{0\} \subseteq \Z^n$   and $\lk(x)^C \subseteq  [-1, 1]^{n-1} \times \{-1\} \subseteq \Z^n$.  Add $x$ as the point $x = \textbf{e}_n = (0, \ldots, 0, 1)$.  This point is adjacent to every point in  $[-1, 1]^{n-1} \times \{0\} \subseteq \Z^n$, and hence to every point of $\lk(x)$ as we have embedded it.  Furthermore,  $x = \textbf{e}_n$ is not adjacent to any point of  $[-1, 1]^{n-1} \times \{-1\} \subseteq \Z^n$, and so this produces exactly the adjacencies of $x$ from $G'$.  This completes the induction. 
\end{proof}

\begin{example}\label{ex: projective plane}
As announced above, we now construct a digital image $X$ that may be viewed as a digital version of the real projective plane $\R P^2$.    Start with a suitable triangulation of $\R P^2$.  Notice that some care must be taken here.  For example, the triangulation of $\R P^2$ given in
\cite[Ex.I.6.2]{Mas91} (see Figure 1.13 on p.15 of \cite{Mas91}) is not suitable.  This is because the clique complex of that triangulation, considered as a graph, contains simplices that are not part of the triangulation (the triangulation has ``empty" simplices, and so is not a clique, or flag complex).  For example, with reference to the notation of \cite[Ex.I.6.2]{Mas91} , the $3$-clique $123$ does not correspond to a $2$-simplex of the triangulation. Indeed, the triangulation of \cite[Ex.I.6.2]{Mas91}, considered as a graph,  is actually a complete graph, and so its clique complex would be a $5$-simplex, with contractible spatial realization.  Instead, we may use  the triangulation of $\R P^2$ (represented as the disc with antipodal points of the boundary circle identified) illustrated in \figref{fig:RP2}.
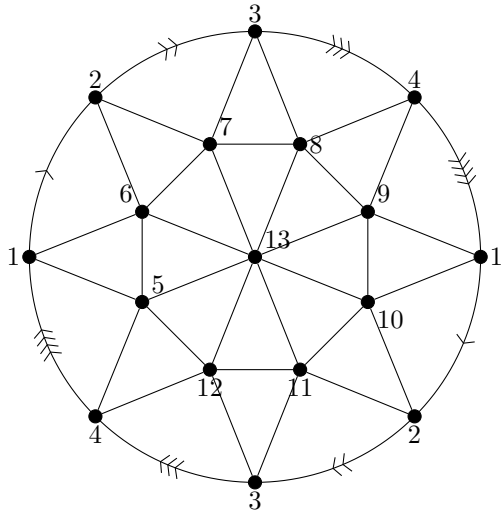
\begin{figure}[h!]
\centering
$$
\begin{tikzpicture}[
    decoration={markings},
    ]


\draw [black,domain=45:90,postaction={decorate,decoration={mark=between positions 0.45 and 0.55 step 0.045 with {\draw (.075,-.13) -- (0,0) -- (.075,.13) ;}}}] plot ({3*cos(\x)}, {3*sin(\x)});

\draw [black,domain=0:45,postaction={decorate,decoration={mark=between positions 0.42 and 0.58 step 0.045 with {\draw (.075,-.13) -- (0,0) -- (.075,.13) ;}}}] plot ({3*cos(\x)}, {3*sin(\x)});

\draw [black,domain=-45:0,postaction={decorate,decoration={mark=between positions 0.5 and 0.5 step 0.1 with {\draw (.075,-.13) -- (0,0) -- (.075,.13) ;}}}] plot ({3*cos(\x)}, {3*sin(\x)});

\draw [black,domain=-90:-45,postaction={decorate,decoration={mark=between positions 0.45 and 0.55 step 0.06 with {\draw (.075,-.13) -- (0,0) -- (.075,.13) ;}}}] plot ({3*cos(\x)}, {3*sin(\x)});

\draw [black,domain=-135:-90,postaction={decorate,decoration={mark=between positions 0.45 and 0.55 step 0.045 with {\draw (.075,-.13) -- (0,0) -- (.075,.13) ;}}}] plot ({3*cos(\x)}, {3*sin(\x)});

\draw [black,domain=-180:-135,postaction={decorate,decoration={mark=between positions 0.42 and 0.58 step 0.045 with {\draw (.075,-.13) -- (0,0) -- (.075,.13) ;}}}] plot ({3*cos(\x)}, {3*sin(\x)});

\draw [black,domain=135:180,postaction={decorate,decoration={mark=between positions 0.5 and 0.5 step 0.1 with {\draw (.075,-.13) -- (0,0) -- (.075,.13) ;}}}] plot ({3*cos(\x)}, {3*sin(\x)});

\draw [black,domain=90:135,postaction={decorate,decoration={mark=between positions 0.45 and 0.55 step 0.06 with {\draw (.075,-.13) -- (0,0) -- (.075,.13) ;}}}] plot ({3*cos(\x)}, {3*sin(\x)});

 \node[inner sep=1.75pt, circle, fill=black] (a) at (0,3) [draw] {};
 \node[inner sep=1.75pt, circle, fill=black] (g) at (-3,0) [draw] {};
 \node[inner sep=1.75pt, circle, fill=black] (b) at (3/1.414,3/1.414) [draw] {};
 \node[inner sep=1.75pt, circle, fill=black] (h) at (-3/1.414,3/1.414) [draw] {};
 \node[inner sep=1.75pt, circle, fill=black] (d) at (3/1.414,-3/1.414) [draw] {};
 \node[inner sep=1.75pt, circle, fill=black] (f) at (-3/1.414,-3/1.414) [draw] {};
 \node[inner sep=1.75pt, circle, fill=black] (e) at (0,-3) [draw] {};
 \node[inner sep=1.75pt, circle, fill=black] (c) at (3,0) [draw] {};
 \node[inner sep=1.75pt, circle, fill=black] (i) at (0,0) [draw] {};

\node[inner sep=1.75pt, circle, fill=black] (v1) at (-.6,1.5) [draw] {};
\node[inner sep=1.75pt, circle, fill=black] (v2) at (.6,1.5) [draw] {};
\node[inner sep=1.75pt, circle, fill=black] (v3) at (1.5,.6) [draw] {};
\node[inner sep=1.75pt, circle, fill=black] (v4) at (1.5,-.6) [draw] {};
\node[inner sep=1.75pt, circle, fill=black] (v5) at (.6,-1.5) [draw] {};
\node[inner sep=1.75pt, circle, fill=black] (v6) at (-.6,-1.5) [draw] {};
\node[inner sep=1.75pt, circle, fill=black] (v7) at (-1.5,-.6) [draw] {};
\node[inner sep=1.75pt, circle, fill=black] (v8) at (-1.5,.6) [draw] {};

%
%

\draw (v1) -- (h);
\draw (v1) -- (a);
\draw (v1) -- (v2);
\draw (v1) -- (v8);
\draw (v1) -- (i);

\draw (v2) -- (i);
\draw (v2) -- (b);
\draw (v2) -- (v3);
\draw (v2) -- (a);

\draw (v3) -- (b);
\draw (v3) -- (c);
\draw (v3) -- (i);
\draw (v3) -- (v4);

\draw (v4) -- (i);
\draw (v4) -- (d);
\draw (v4) -- (c);
\draw (v4) -- (v5);

\draw (v5) -- (i);
\draw (v5) -- (d);
\draw (v5) -- (e);
\draw (v5) -- (v6);

\draw (v6) -- (i);
\draw (v6) -- (e);
\draw (v6) -- (f);
\draw (v6) -- (v7);

\draw (v7) -- (i);
\draw (v7) -- (f);
\draw (v7) -- (g);
\draw (v7) -- (v8);

\draw (v8) -- (i);
\draw (v8) -- (g);
\draw (v8) -- (h);

\node[anchor = south ]  at (a) {{$3$}};
\node[anchor = south ]  at (b) {{$4$}};
\node[anchor = west ]  at (c) {{$1$}};
\node[anchor = north ]  at (d) {{$2$}};
\node[anchor = north ]  at (e) {{$3$}};
\node[anchor = north ]  at (f) {{$4$}};
\node[anchor = east ]  at (g) {{$1$}};
\node[anchor = south ]  at (h) {{$2$}};
\node[anchor = south west]  at (i) {{$13$}};

\node[anchor = south west ]  at (v1) {{$7$}};
\node[anchor =  west ]  at (v2) {{$8$}};
\node[anchor = south west ]  at (v3) {{$9$}};
\node[anchor = north west ]  at (v4) {{$10$}};
\node[anchor = north ]  at (v5) {{$11$}};
\node[anchor = north ]  at (v6) {{$12$}};
\node[anchor = south west ]  at (v7) {{$5$}};
\node[anchor = south east ]  at (v8) {{$6$}};

\end{tikzpicture}
$$
\caption{Triangulation of $\R P^2$}\label{fig:RP2}
\end{figure}
Observe that this triangulation, considered as a graph (after making the identifications indicated), contains $3$-cliques, each of which corresponds to a $2$-simplex of the triangulation, and does not contain any $4$-cliques.  Therefore, if $G$ is the (abstract) graph, or tolerance space illustrated, its clique complex will give $\cl(G) = K$, where $K$ is the (abstract) simplicial complex indicated, and thus  $|\cl(G)|$ will be homeomorphic to  $\R P^2$.  

It remains to display the abstract graph/tolerance space $G$ as a digital image, up to isomorphism.  
\propref{prop: digital graph} provides a general scheme for doing this which, if followed strictly, would result in a digital image in $\Z^{12}$.  We may adapt that scheme here and get off to a more efficient start (in terms of embedding dimension) by embedding $8$ vertices of $G$ in $\Z^3$.  Remove the vertex $13$ from $G$.  Observe that, if  the identifications indicated are made now, we would obtain a triangulated M{\"o}bius strip.  Then the vertex $13$ is a cone-point on the boundary of this M{\"o}bius strip.  Topologically, this is one way to see that $\R P^2$ may be embedded in $\R^4$. However, it seems  that, here, our particular triangulation of the M{\"o}bius strip does not embed in $\Z^4$ as a digital image.  So remove also the vertices $1$, $2$, $3$, and $4$.  What remains is an $8$-point cycle graph, which we may embed as a digital image in $\Z^3$.  In fact, we may embed the $8$-point cycle graph as a digital image in the cube $[-1, 1]^3 \subseteq \Z^3$.  The coordinates of the $8$ vertices of this cycle graph may be assigned as follows:
$$
\begin{aligned}
5 &= (1, 0, 1)  \quad 6 = (1, 1, 0)  \quad 7 = (0, 1, -1) \quad 8 = (-1, 1, 0)\\
9 &= (-1, 0, 1)  \quad 10 = (-1, -1, 0)  \quad 11 = (0, -1, -1) \quad 12 = (1, -1, 0)\\
\end{aligned}
$$
Since this is a digital image in $[-1, 1]^3$, we may now proceed with the general scheme of \propref{prop: digital graph} for embedding a graph as a digital image.  The result will be $G$ embedded as a digital image $X$ in $[-1, 1]^8 \subseteq \Z^8$.  We add the vertices $1$, $2$, $3$, $4$, and $13$, in that order, and as we add each vertex we preserve the adjacencies amongst prior vertices and add the adjacencies between them and the vertex being added.

Add vertex $1$:  Embed the graph thus far into $[-1, 1]^3 \times \{0\} \subseteq \Z^4$; move the last coordinate of those vertices not adjacent to vertex $1$ to $-1$; add the vertex $1$ as $(0, 0, 0, 1)$.  This results in
$$
\begin{aligned}
5 &= (1, 0, 1, 0)  \quad 6 = (1, 1, 0, 0)  \quad 7 = (0, 1, -1, -1) \quad 8 = (-1, 1, 0, -1)\\
9 &= (-1, 0, 1, 0)  \quad 10 = (-1, -1, 0, 0)  \quad 11 = (0, -1, -1, -1) \quad 12 = (1, -1, 0, -1)\\
1 &= (0, 0, 0, 1).
\end{aligned}
$$
The next three steps repeat this process, following the scheme of \propref{prop: digital graph}.  These steps result in a digital image in $\Z^7$ with points
$$
\begin{aligned}
5 &= (1, 0, 1, 0, -1, -1, 0)  \quad 6 = (1, 1, 0, 0, 0, -1, -1)  \quad 7 = (0, 1, -1, -1, 0, 0, -1)\\
8 &= (-1, 1, 0, -1, -1, 0, 0) \quad 9 = (-1, 0, 1, 0, -1, -1, 0)  \quad 10 = (-1, -1, 0, 0, 0, -1, -1)\\
11 &= (0, -1, -1, -1, 0, 0, -1) \quad 12 = (1, -1, 0, -1, -1, 0, 0) \quad 1 = (0, 0, 0, 1, 0, -1, 0) \\
2 &= (0, 0, 0, 0, 1, 0, -1) \quad 3 = (0, 0, 0, 0, 0, 1, 0) \quad 4 = (0, 0, 0, 0, 0, 0, 1).
\end{aligned}
$$
Finally, we add the vertex $13$, using the same scheme.  This is the point that corresponds to the cone-point if we visualize projective space as the  M{\"o}bius strip with a cone attached to its boundary.  The result is the digital image $X \subseteq \Z^8$ consisting of the $13$ points
$$
\begin{aligned}
5 &= (1, 0, 1, 0, -1, -1, 0, 0)  \quad 6 = (1, 1, 0, 0, 0, -1, -1, 0)  \quad 7 = (0, 1, -1, -1, 0, 0, -1, 0)\\
8 &= (-1, 1, 0, -1, -1, 0, 0, 0) \quad 9 = (-1, 0, 1, 0, -1, -1, 0, 0) \quad 10 = (-1, -1, 0, 0, 0, -1, -1, 0)\\
11&= (0, -1, -1, -1, 0, 0, -1, 0) \quad 12 = (1, -1, 0, -1, -1, 0, 0, 0) \quad 1 = (0, 0, 0, 1, 0, -1, 0, -1) \\
2 &= (0, 0, 0, 0, 1, 0, -1, -1) \quad 3 = (0, 0, 0, 0, 0, 1, 0, -1) \quad 4 = (0, 0, 0, 0, 0, 0, 1, -1)\\
13& =  (0, 0, 0, 0, 0, 0, 0, 1).
\end{aligned}
$$
As a digital image, recall, it is not necessary to specify adjacencies: these are determined by position, or coordinates, in $\Z^8$.

For this digital image $X$, by construction, we have $\cl(X)$ isomorphic to the complex represented by $G$, as a simplicial complex, and thus the spatial realization $|\cl(X)|$ is homeomorphic to $\R P^2$. As is well-known, we have $\pi_1(\R P^2; x_0) \cong \Z_2$ (see \cite[Ex.V.5.2]{Mas91}, for example). From
\thmref{thm: digital pi1 = edge pi1} and \thmref{thm: edge gp = pi1}, it follows that we have
$$\pi_1(X; x_0) \cong \Z_2.$$
Notice  that it would also be possible to calculate $\pi_1(X; x_0) \cong \Z_2 $ using \corref{cor: contractible V}, mimicking the steps in the argument used for \cite[Ex.V.5.2]{Mas91}.
This example illustrates that a digital image may have torsion in its fundamental group.
\end{example}

Finally, for this section, we use the approach of \exref{ex: projective plane} to show the following general realization result.

\begin{theorem}\label{thm: fg realization}
Every finitely presented group occurs as the (digital) fundamental group of some digital image.
\end{theorem}

\begin{proof}
Suppose $G$ is a finitely presented group with finite presentation
$$G = \langle g_1,  \ldots, g_n \mid R_1, \ldots, R_m \rangle.$$
Here, each $R_j$ is a word in the $g_i$ and their inverses $g_i^{-1}$.  We may suppose these words are in reduced form (no occurrences of a generator juxtaposed with its own inverse).  First we build in the usual way, but taking care to avoid empty simplices, a two-dimensional simplicial complex with this $G$ as edge group.   For the one-skeleton, take an $n$-fold one-point union of length-$4$ cycle graphs with vertices
$$V = \{ v_0 \} \cup \bigcup_{i=1}^n \{ v_{i, 1}, v_{i, 2}, v_{i, 3} \}$$
and edges
$$E = \bigcup_{i=1}^n \left\{   \{ v_0, v_{i, 1}\}, \{ v_{i, 1}, v_{i, 2} \}, \{ v_{i, 2}, v_{i, 3} \}, \{ v_{i, 3}, v_0 \} \right\}.$$
The case in which $n = 2$ is illustrated in \figref{fig:n=2} below.
The edge group of this graph is the free group on $n$ generators, which we may identify with the free group $\langle g_1,  \ldots, g_n \rangle$ in an obvious way.  Namely, each generator $g_i$ corresponds to the edge loop  $\{ v_0, v_{i, 1}, v_{i, 2}, v_{i, 3}, v_0 \}$  of length $4$.  The inverse of a generator corresponds to the reverse path:  $g_i^{-1}$ corresponds to the edge loop  $\{ v_0, v_{i, 3}, v_{i, 2}, v_{i, 1}, v_0 \}$.  Because each of the generating cycle graphs is of length four, there are no $3$-cliques in this graph, hence no empty $2$-simplices.

Next, for each relator $R_j$, we wish to attach a (triangulated) disk so as to  introduce this relation into the edge group.   Here, again, we just have to be careful not to introduce any empty $3$-simplices.
We may achieve this as follows.  Consider a single relator $R$.  Suppose $R$ is a word
$$R = g_{j_1}^{\epsilon_1} \cdots g_{j_k}^{\epsilon_k}$$
of length $k$ in the letters $\{ g_i, g_i^{-1} \}$, with each $\epsilon_r$ either $1$ or $-1$.  Define a cycle graph $C$ of length $4k$ whose vertices we list in order as
$$V_C = \{ w_1, w_{1, 1}, w_{1, 2}, w_{1, 3},  w_2, w_{2, 1}, w_{2, 2}, w_{2, 3}, w_3, \ldots,  w_k, w_{k, 1},  w_{k, 2}, w_{k, 3} \},$$
with adjacent vertices of this list joined by an edge of $C$, as well as the last vertex $w_{k, 3}$ and the first vertex $w_1$ joined by an edge.  Take a copy of this cycle graph $C'$ with vertices
$$V_{C'} = \{ w'_1, w'_{1, 1}, w'_{1, 2}, w'_{1, 3},  w'_2, w'_{2, 1}, w'_{2, 2}, w'_{2, 3}, w'_3, \ldots,  w'_k, w'_{k, 1}, w'_{k, 2}, w'_{k, 3} \}.$$
Now join the $i$th listed  vertex of $C$ to the $i$th and $(i+1)$st listed vertices of $C'$ (treating the $(4k+1)$st as the first).  This creates a ``triangulated annulus," with $C$ as outer boundary and $C'$ as inner boundary.  Finally, add another vertex $w$, and join this vertex to every vertex of $C'$.  A case in which $k=3$ is illustrated in \exref{ex: fg example} below (see \figref{fig:disk}).

So far, we have built a triangulated disk that has no $4$-cliques. Now attach this disk to the one-point union of length-$4$ cycle graphs, according as the letters of the relator $R$. Namely, identify for each $i$ the edge loops (vertex-for-vertex and edge-for-edge)
$$w_i, w_{i, 1}, w_{i, 2}, w_{i, 3},  w_{i+1} \text{ with } \begin{cases} v_0, v_{i, 1}, v_{i, 2}, v_{i, 3},  v_0 & \text{ if }  \epsilon_i = 1\\
v_0, v_{i, 3}, v_{i, 2}, v_{i, 1},  v_0 & \text{ if }  \epsilon_i = -1.\end{cases}$$
Now it is standard that attaching this disk in this way introduces the relation $R$ into the edge group (and no other relations).  The main point here, though, is that we have introduced the desired relation by building a $2$-dimensional simplicial complex that has no empty $2$-simplices, and no $4$-cliques (hence no empty  $3$- or higher simplices).  Considering the $1$-skeleton of the complex after attaching the disk as a graph,  its clique complex is the $2$-dimensional complex we have constructed.

It is clear that we may apply this last step to each of  the relators $R_j$.  Doing so constructs a $2$-dimensional simplicial complex $K$, that is the original one-point union of $n$ length-$4$ cycle graphs, with $m$ triangulated disks attached as in the step above.  The edge group of $K$, by construction, is $G$.  As a (finite, simple) graph, we may embed the one-skeleton of $K$ into some $\Z^n$ (possibly a high-dimensional such) as a digital image, following the scheme of \propref{prop: digital graph}.  Furthermore, from the way in which we have constructed and attached the triangulated disks, the clique complex of this digital image, considered as the graph we started from, is exactly $K$.     Then the (digital) fundamental group of this digital image is $G$, as follows from \thmref{thm: digital pi1 = edge pi1}.
\end{proof}

\begin{example}\label{ex: fg example}
We illustrate the above result with an example.  Take
$$G = \langle g_1, g_2 \mid g_1 g_2 g_1^{-1} \rangle.$$
Following the recipe of the proof of \thmref{thm: fg realization}, we start with a graph that is a one-point union of two cycle graphs of length $4$:
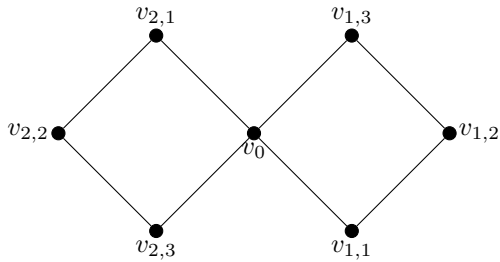
\begin{figure}[h!]
\centering

$$
\begin{tikzpicture}[scale=1.3,
    decoration={markings},
    ]

\node[inner sep=1.75pt, circle, fill=black] (v0) at (0,0) [draw] {};
\node[inner sep=1.75pt, circle, fill=black] (v21) at (-1,1) [draw] {};
\node[inner sep=1.75pt, circle, fill=black] (v22) at (-2,0) [draw] {};
\node[inner sep=1.75pt, circle, fill=black] (v23) at (-1,-1) [draw] {};
\node[inner sep=1.75pt, circle, fill=black] (v13) at (1,1) [draw] {};
\node[inner sep=1.75pt, circle, fill=black] (v12) at (2,0) [draw] {};
\node[inner sep=1.75pt, circle, fill=black] (v11) at (1,-1) [draw] {};

\draw (v0) -- (v21);
\draw (v0) -- (v23);
\draw (v0) -- (v13);
\draw (v0) -- (v11);

\draw (v22) -- (v21);
\draw (v22) -- (v23);

\draw (v12) -- (v13);
\draw (v12) -- (v11);

\node[anchor =  north ]  at (v0) {{$v_{0}$}};
\node[anchor =  south ]  at (v21) {{$v_{2,1}$}};
\node[anchor =  east ]  at (v22) {{$v_{2,2}$}};
\node[anchor =  north ]  at (v23) {{$v_{2,3}$}};
\node[anchor =  south ]  at (v13) {{$v_{1,3}$}};
\node[anchor =  west ]  at (v12) {{$v_{1,2}$}};
\node[anchor =  north ]  at (v11) {{$v_{1,1}$}};

\end{tikzpicture}
$$
\caption{Two-fold one-point union of cycle graphs of length $4$.  }\label{fig:n=2}
\end{figure}

Next, we construct a triangulated disk whose boundary corresponds to the relation we wish to introduce.  Once again following the recipe of the proof of \thmref{thm: fg realization}, this will consist of:  a cycle graph of length $12$; an intermediate cycle graph of  the same length; an evident triangulation of the ``annulus" with these cycle graphs as boundary; a cone-point added to ``cone-off" the inner cycle graph.  In \figref{fig:disk}, we have illustrated the result, and also indicated the identifications we make along the boundary, with vertices and edges identified with their counterparts in the one-point union illustrated above.
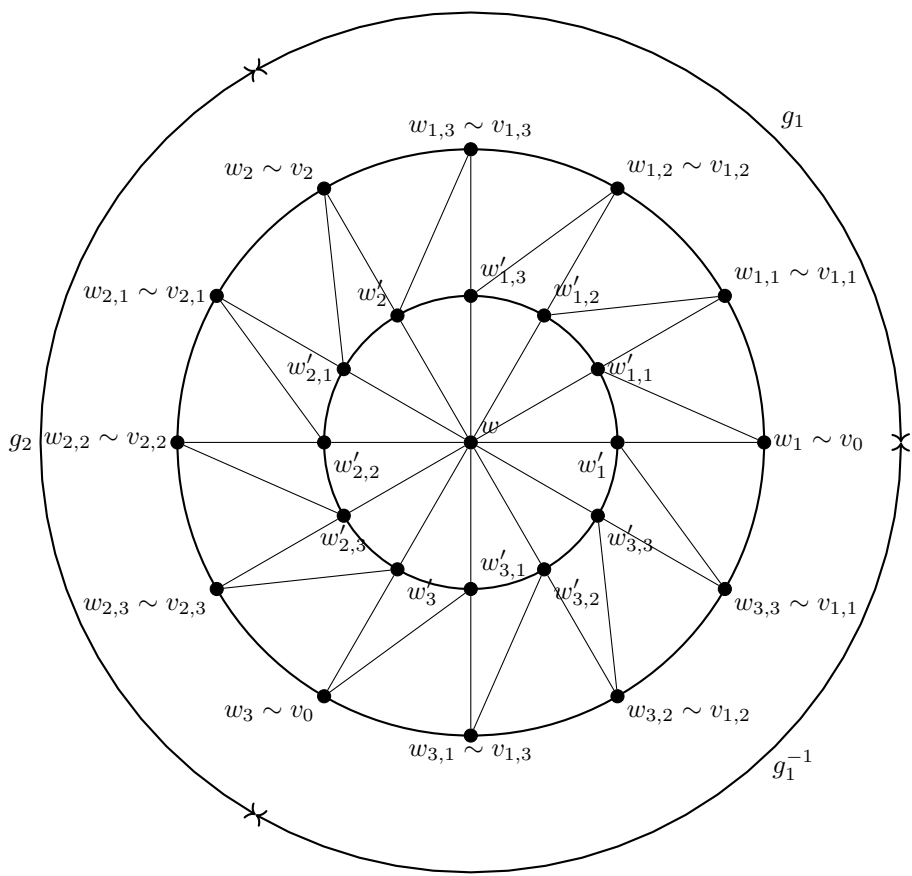
\begin{figure}[h!]
\centering

$$
\begin{tikzpicture}[scale=1.3,
    decoration={markings},
    ]
\draw[black, thick] (0,0) circle [radius=3];
\draw[black, thick] (0,0) circle [radius=1.5];
 \draw [thick,domain=0:120, <->] plot ({4.4*cos(\x)}, {4.4*sin(\x)});
 \draw [thick,domain=120:240, <->] plot ({4.4*cos(\x)}, {4.4*sin(\x)});
  \draw [thick,domain=240:360, <->] plot ({4.4*cos(\x)}, {4.4*sin(\x)});

\node[]  at (3.3,3.3) {{$g_1$}};
\node[]  at (-4.6,0) {{$g_2$}};
\node[]  at (3.3,-3.3) {{$g^{-1}_1$}};

 \node[inner sep=1.75pt, circle, fill=black] (w) at (0,0) [draw] {};

 \node[inner sep=1.75pt, circle, fill=black] (w1) at (3,0) [draw] {};
 \node[inner sep=1.75pt, circle, fill=black] (w2) at (2.598,1.5) [draw] {};
 \node[inner sep=1.75pt, circle, fill=black] (w3) at (1.5,2.598) [draw] {};
 \node[inner sep=1.75pt, circle, fill=black] (w4) at (0,3) [draw] {};
 \node[inner sep=1.75pt, circle, fill=black] (w5) at (-1.5,2.598) [draw] {};
 \node[inner sep=1.75pt, circle, fill=black] (w6) at (-2.598,1.5) [draw] {};
 \node[inner sep=1.75pt, circle, fill=black] (w7) at (-3,0) [draw] {};
 \node[inner sep=1.75pt, circle, fill=black] (w8) at (-2.598,-1.5) [draw] {};
 \node[inner sep=1.75pt, circle, fill=black] (w9) at (-1.5,-2.598) [draw] {};
  \node[inner sep=1.75pt, circle, fill=black] (w10) at (0,-3) [draw] {};
   \node[inner sep=1.75pt, circle, fill=black] (w11) at (1.5,-2.598) [draw] {};
    \node[inner sep=1.75pt, circle, fill=black] (w12) at (2.598,-1.5) [draw] {};

 \node[inner sep=1.75pt, circle, fill=black] (w1') at (3/2,0) [draw] {};
 \node[inner sep=1.75pt, circle, fill=black] (w2') at (2.598/2,1.5/2) [draw] {};
 \node[inner sep=1.75pt, circle, fill=black] (w3') at (1.5/2,2.598/2) [draw] {};
 \node[inner sep=1.75pt, circle, fill=black] (w4') at (0,3/2) [draw] {};
 \node[inner sep=1.75pt, circle, fill=black] (w5') at (-1.5/2,2.598/2) [draw] {};
 \node[inner sep=1.75pt, circle, fill=black] (w6') at (-2.598/2,1.5/2) [draw] {};
 \node[inner sep=1.75pt, circle, fill=black] (w7') at (-3/2,0) [draw] {};
 \node[inner sep=1.75pt, circle, fill=black] (w8') at (-2.598/2,-1.5/2) [draw] {};
 \node[inner sep=1.75pt, circle, fill=black] (w9') at (-1.5/2,-2.598/2) [draw] {};
  \node[inner sep=1.75pt, circle, fill=black] (w10') at (0,-3/2) [draw] {};
   \node[inner sep=1.75pt, circle, fill=black] (w11') at (1.5/2,-2.598/2) [draw] {};
    \node[inner sep=1.75pt, circle, fill=black] (w12') at (2.598/2,-1.5/2) [draw] {};

\draw (w) -- (w1');
\draw (w) -- (w2');
\draw (w) -- (w3');
\draw (w) -- (w4');
\draw (w) -- (w5');
\draw (w) -- (w6');
\draw (w) -- (w7');
\draw (w) -- (w8');
\draw (w) -- (w9');
\draw (w) -- (w10');
\draw (w) -- (w11');
\draw (w) -- (w12');

\draw (w1') -- (w12);
\draw (w1') -- (w1);

\draw (w2') -- (w1);
\draw (w2') -- (w2);

\draw (w3') -- (w2);
\draw (w3') -- (w3);

\draw (w4') -- (w3);
\draw (w4') -- (w4);

\draw (w5') -- (w4);
\draw (w5') -- (w5);

\draw (w6') -- (w5);
\draw (w6') -- (w6);

\draw (w7') -- (w6);
\draw (w7') -- (w7);

\draw (w8') -- (w7);
\draw (w8') -- (w8);

\draw (w9') -- (w8);
\draw (w9') -- (w9);

\draw (w10') -- (w9);
\draw (w10') -- (w10);

\draw (w11') -- (w10);
\draw (w11') -- (w11);

\draw (w12') -- (w11);
\draw (w12') -- (w12);

\node[anchor = south west ]  at (w) {{$w$}};

\node[anchor =  north east ]  at (w1') {{$w_1'$}};
\node[anchor =  west ]  at (w2') {{$w_{1,1}'$}};
\node[anchor = south west ]  at (w3') {{$w_{1,2}'$}};
\node[anchor = south west ]  at (w4') {{$w_{1,3}'$}};
\node[anchor = south east ]  at (w5') {{$w_2'$}};
\node[anchor = east ]  at (w6') {{$w_{2,1}'$}};
\node[anchor = north west ]  at (w7') {{$w_{2,2}'$}};
\node[anchor = north ]  at (w8') {{$w_{2,3}'$}};
\node[anchor = north west ]  at (w9') {{$w_3'$}};
\node[anchor = south west ]  at (w10') {{$w_{3,1}'$}};
\node[anchor = north west ]  at (w11') {{$w_{3,2}'$}};
\node[anchor = north west ]  at (w12') {{$w_{3,3}'$}};

\node[anchor =  west ]  at (w1) {{$w_{1}\sim v_{0}$}};
\node[anchor = south west ]  at (w2) {{$w_{1,1}\sim v_{1,1}$}};
\node[anchor = south west ]  at (w3) {{$w_{1,2}\sim v_{1,2}$}};
\node[anchor =  south  ]  at (w4) {{$w_{1,3}\sim v_{1,3}$}};
\node[anchor = south east ]  at (w5) {{$w_{2}\sim v_{2}$}};
\node[anchor =  east ]  at (w6) {{$w_{2,1}\sim v_{2,1}$}};
\node[anchor =  east ]  at (w7) {{$w_{2,2}\sim v_{2,2}$}};
\node[anchor = north east ]  at (w8) {{$w_{2,3}\sim v_{2,3}$}};
\node[anchor = north east ]  at (w9) {{$w_{3}\sim v_{0}$}};
\node[anchor = north  ]  at (w10) {{$w_{3,1}\sim v_{1,3}$}};
\node[anchor = north west ]  at (w11) {{$w_{3,2}\sim v_{1,2}$}};
\node[anchor = north west ]  at (w12) {{$w_{3,3}\sim v_{1,1}$}};

\end{tikzpicture}
$$
\caption{Triangulated disk, plus attachements.  }\label{fig:disk}
\end{figure}
Identifying the boundary of this triangulated disk, in the way indicated, to the one-point union of length-$4$ cycle graphs illustrated in  \figref{fig:n=2} results in a $2$-dimensional simplicial complex whose edge group is $G$.    This simplicial complex has $7+12+1=20$ vertices.  Following our general scheme for embedding a graph as a digital image, we may realize the one-skeleton of this complex as a digital image in some $\Z^n$ with $n \leq 19$ (considerably less should be possible).  Furthermore, the clique complex of this digital image, considered as the graph that we realized,  has clique complex exactly  this simplicial complex, with edge group $G$.  This digital image realizes the group $G$.
\end{example}

\begin{remark}
\thmref{thm: fg realization}, \propref{prop: digital graph},  \exref{ex: projective plane} and \exref{ex: fg example} taken together raise interesting questions.  First, is it the case that every homotopy type may be taken as the spatial realization of a simplicial complex  that is a clique complex?  As we saw in the above example, in some cases at least, triangulations commonly used to represent a space as a simplicial complex need not be clique complexes.  Second, when we do have a homotopy type represented as the spatial realization of some clique complex $\cl(G)$,  we may always display $G$ as a digital image, but the embedding dimension may be quite high.  It would interesting, for example, to know whether it is possible to have a digital image in $\Z^4$ whose clique complex has spatial realization homeomorphic to $\R P^2$.  Generally speaking, even when we have a graph $G$ whose $E(G, v_0)$ gives some group of interest, it does not seem easy to determine the minimal embedding dimension of $G$ as a digital image.  For instance, it is not immediately clear which groups might be obtained as the fundamental groups of 3D digital images.
\end{remark}

\section{Path Shortening and 2D Digital Images}\label{sec: 2D free}

Whilst \thmref{thm: digital pi1 = edge pi1} and  \thmref{thm: edge gp = pi1} allow us to use many results from the topological setting in  the digital setting, they do not automatically resolve all questions about the digital fundamental group.  For example, as just remarked,  it is not immediately clear which groups might be obtained as the fundamental groups of 3D digital images.
Likewise the digital fundamental group of a general 2D image.
In fact we will show in Theorem \ref{thm: pi is free} below that the fundamental group of every 2D digital image is a free group.   Now, the clique complex of a 2D image, generally speaking, is a simplicial complex with simplices of dimension up to $3$.  There is no general reason why such a simplicial complex should have fundamental group that is a free group.  So some argument is required, either in the digital setting or, using  \thmref{thm: digital pi1 = edge pi1} and  \thmref{thm: edge gp = pi1}, in the simplicial complex setting or in the topological setting.  We argue in the digital setting.
To prepare for this result, we establish some basic results about paths and digital circles.

\begin{definition}
Let $X \subseteq \Z^r$ be any digital image.  Suppose we have two points $a, b \in X$ that are non-adjacent.  We say that a set of $n+2$ (distinct) points $P = \{ a, x_1, \ldots, x_n, b\} \subseteq X$ with $n \geq 1$ is a \emph{contractible path in $X$ from $a$ to $b$} of length $n+1$ if we have adjacencies $a \sim_X x_1$, $x_i \sim_X x_{i+1}$ for each $1 \leq i \leq n-1$, and $x_n \sim_X b$, and no other adjacencies amongst the elements of $P$.
\end{definition}

The relationship on pairs of points of having a contractible path from one to the other is clearly symmetric: a contractible path from $a$ to $b$ will serve as a contractible path from $b$ to $a$.
The nomenclature is justified by the following observations.

\begin{lemma}
Suppose we have a set of points $P = \{ a, x_1, \ldots, x_n, b\} \subseteq X$ that  is a \emph{contractible path in $X$ from $a$ to $b$}.
\begin{itemize}
\item[(A)] There is a path $\alpha\colon I_{n+1} \to X$ with $\alpha(0) = a$, $\alpha(n+1) = b$, and $\alpha(i) = x_i$ for $1 \leq i \leq n$.
\item[(B)] This path gives an isomorphism of digital images $I_{n+1} \cong P$
\item[(C)] With $a \in P$ as basepoint, $P$ is a based-contractible subset of $X$ (contractible in itself, not just in $X$).
\end{itemize}
\end{lemma}

\begin{proof}
(A) This point is more or less tautological.  We just need to observe that $\alpha$ as defined is continuous, which is to say that we have $\alpha(i) \sim_X \alpha(i+1)$ for each $0 \leq i \leq n$.  This is part of the data given about $P$.

(B) The path $\alpha \colon I_{n+1} \to P$ has continuous inverse $g\colon P \to I_{n+1}$ given by $g(a) = 0$, $g(n+1) = b$, and $g(x_i) = i$ for $1 \leq i \leq n$.  Notice that this depends on the points of $P$ being distinct from each other (no repeats).

(C)  An interval is based-contractible, via a based contracting homotopy, to any of its points.   In Example 3.13 of \cite{LOS19c}, for example, we give a contracting homotopy $H \colon I_{n+1} \times I_{n+1} \to I_{n+1}$ that satisfies $H(i, 0) = i$ and $H(i, n+1) = 0$, and is a based homotopy in the sense that we also have $H(0, t) = 0$ for all $t \in I_{n+1}$.  The homotopy is defined by
$$H(i, t) = \begin{cases} i  & 0 \leq i \leq n+1-t \\
n+1-t & n+2-t \leq i \leq n+1.\end{cases}$$
This evidently satisfies $H(i, 0) = i$,  $H(i, n+1) = 0$, and $H(0, t) = 0$.  The only issue is whether $H$ is continuous.
Since we omitted the details of the check on continuity in Example 3.13 of \cite{LOS19c} (and also in Example 3.19 of  \cite{LOS19a}), we provide the details here.

To check continuity, suppose that we have $(i, t) \sim_{I_{n+1} \times I_{n+1}} (i', t')$.  We must show that $H(i, t) \sim_{I_{n+1}} H(i', t')$.
If the coordinates $(x, y)$ of both points satisfy $x+y \leq n+1$, then we have $i \leq n+1-t$ and $i' \leq n+1-t'$, and the formula for $H$ gives $| H(i', t') - H(i, t)| = |i' - i| \leq 1$, since $(i, t) \sim (i', t')$ means that we have  $|i' - i| \leq 1$ and $|t' - t| \leq 1$.
If the coordinates $(x, y)$ of both points satisfy $x+y \geq n+1$, then $| H(i', t') - H(i, t)| = |(n+1-t') - (n+1-t)| = |(t-t')| \leq 1$, again because $(i, t) \sim (i', t')$.  The only case that remains, then, is that in which  the coordinates of one point  satisfy $x+y \leq n$ and those of  the other point satisfy  $x+y \geq n+2$.  Since $(i, t) \sim (i', t')$ entails $| (i'+t') - (i+t)| \leq 2$, we must have $i+t = n$ and $i'+t' = n+2$. (There is no loss of generality in writing the point to the lower-left of  the other as $(s, t)$.)  But then we have $t' -t =1$ (as well as $i'-i = 1$), from the adjacency  $(i, t) \sim (i', t')$.  It follows  that  $| H(i', t') - H(i, t)| = |(n+1-t') - i| = |(n+1-t') - (n-t)|= |1-(t'-t)| = 0$.  In all cases, we have  $H(i, t) \sim_{I_{n+1}} H(i', t')$, so $H$ is indeed continuous.

Now contractibility is preserved by an isomorphism of digital images (it is also preserved by other, much more general notions of ``same-ness").  Here, the isomorphisms $\alpha$ and $g$ of part (B) define a homotopy
$$G = \alpha\circ H \circ (g \times \mathrm{id}_{I_{n+1}}) \colon P \times I_{n+1} \to P,$$
that satisfies  $G(p, 0) = \alpha\circ g(p) = p$ and $G(p, n+1) = \alpha(0) = a$ for each $p \in P$.  The homotopy $G$ also satisfies $G(a, t) = \alpha\circ H(0, t) = \alpha(0) = a$ for each $t \in I_{n+1}$, so it is a based contraction of $P$ in the sense asserted.
\end{proof}

If we remove a point from a digital circle, we obtain a contractible path.  Whilst there are many contractible paths that may be ``completed" to a digital circle by the addition of a suitable point, there are examples of contractible paths that may not be completed to a circle, even when we have $a$ and $b$ adjacent to a common point of $X$.

\begin{example}
Take $X\subseteq \Z^2$ by $X=\{(-1,0), (0,0), (1,0)\}$. Then $X$ is a contractible path that cannot be completed to a digital circle by a single point.  This is because the only points not in $X$ adjacent to $a=(-1,0)$ and $b=(1,0)$ are $(0,1)$ and $(0,-1)$, which are adjacent to $(0,0)$.
\end{example}

We have the following ``shortening lemma."

\begin{lemma}\label{lem: contract path}
Let $X \subseteq \Z^n$ be any digital image.   For non-adjacent points $a, b \in X$, if there is path in $X$ from $a$ and $b$, then the path may be shortened to a contractible path in $X$ from $a$ to $b$.
\end{lemma}

\begin{proof}
Suppose we have a path $\gamma\colon I_N \to X$ with $\gamma(0) = a$ and $\gamma(N) = b$.  If we have $\gamma(i) = \gamma(i+k)$ for some $k \geq 1$ and $0 \leq i \leq N-1$, then we simply delete the part of the path $\gamma$ between the adjacent values (including one of the repeats).    Specifically, we define a shorter path $\gamma'\colon I_{N-k} \to X$ by
$$\gamma'(s) = \begin{cases} \gamma(s) & 0 \leq s \leq i\\
\gamma(s+k) & i+1 \leq s \leq N-k.\end{cases}$$
These two parts of $\gamma$ join to give a continuous $\gamma'$, since at the join we have $\gamma'(i) = \gamma(i) = \gamma(i+ k)$ and $\gamma'(i+1) = \gamma(i+k+1)$, which are adjacent in $X$ by the continuity of $\gamma$.

By repeating the first step sufficiently many times, we may assume without loss of generality that $\gamma'\colon I_{N'} \to X$ is a (shorter) path from $a$ to $b$ that does not have any repeated values.  Suppose  we have $\gamma'(i) \sim_X \gamma'(i+k')$ for some $k' \geq 2$ and $0 \leq i \leq N-2$, then again we simply delete the part of the path $\gamma'$ between the adjacent values (leaving both adjacent values themselves).    Specifically, we define a shorter path $\gamma''\colon I_{N'-k'+1} \to X$ by
$$\gamma''(s) = \begin{cases} \gamma'(s) & 0 \leq s \leq i\\
\gamma(s+k'-1) & i+1 \leq s \leq N-k'+1.\end{cases}$$
These two parts of $\gamma''$ join to give a continuous $\gamma'$, since at the join we have $\gamma''(i) = \gamma'(i)$ and $\gamma''(i+1) = \gamma'(i+k')$, which are adjacent in $X$ by the assumption on  $\gamma'$.  By repeating this step sufficiently many times, we arrive at a path $\gamma''\colon I_{N''} \to X$ from $a$ to $b$ that satisfies
$$\{ \gamma''(i) \mid 0 \leq i \leq N'' \} \subseteq  \{ \gamma(i) \mid 0 \leq i \leq N \},$$
so it is a ``shortening" of the original path from $a$ to $b$.  It also satisfies $\gamma''(i) \not\sim \gamma''(i+k)$ for $k \geq 2$, for each $0 \leq i \leq N'-2$, and does not contain any repeated values.  The set  $\{ \gamma''(i) \mid 0 \leq i \leq N'' \}$ gives a contractible path in $X$ from $a$ to $b$.
\end{proof}

We have one more ingredient to prepare for our main result.  We recall a definition from \cite{LOS19c}.

\begin{definition}[Based Homotopy Equivalence]\label{def: based h.e.}
Let $f \colon X \to Y$ be a based map of based digital  images.  If there is a based map $g \colon Y \to X$ such that $g\circ f \approx \text{id}_X$ and $f \circ g \approx \text{id}_Y$, then $f$ is a \emph{based-homotopy equivalence}, and $X$ and $Y$ are said to be \emph{based-homotopy equivalent}, or to have the same \emph{based-homotopy type}.
\end{definition}

In this definition, the notation ``$\approx$" denotes based homotopy of based maps, as we recalled in \secref{sec: basics}.  As we remarked in \cite{LOS19c}, the notion of based homotopy equivalence of digital images is often too rigid to be of much use as a notion of ``same-ness" for digital images.  However, in the following result, we do find  a use for it.  It follows easily from Lemma 3.11 of \cite{LOS19c} that if $X$ and $Y$ are based-homotopy equivalent digital images, then their digital fundamental groups are isomorphic.

We are now ready to prove the main result of this section.

\begin{theorem}\label{thm: pi is free}
Let $X \in \Z^2$ be a connected 2D digital image.  Then $\pi_1(X; x_0)$ is a free group.
\end{theorem}

\begin{proof}
We argue by induction on the (finite) number of points $N$ in the digital image.  Induction starts with $N= 1, 2,$ or $3$, where there is nothing to prove ($X$ is contractible to a point in these cases, so has $\pi_1(X;x_0) \cong \{\mathbf{e}\}$).

So assume inductively that, for any 2D digital image with $n$ or fewer points, the fundamental group is free.  Now suppose $X$ is a digital image with $n+1$ points.

We may totally order the points of $X$ by lexicographic order.  That is, $(x_1, y_1) > (x_2, y_2)$ if $x_1 > x_2$, and  $(x, y_1) > (x, y_2)$ if $y_1 > y_2$.  Suppose that $x\in X$ is the maximal point in this ordering, so that there are no points of $X$ with a greater first coordinate, and the only points with the same first coordinate as that of $x$ have smaller second coordinate.     The possible neighbours of $x$ in $X$ are illustrated as follows (there are at most $4$ of them):
$$
\begin{tabular}{c|c}
$a$ & \\
\hline \\
$c$ & $x$ \\
\hline \\
$b_1$ & $b_2$ \\
\end{tabular}
$$
The \emph{link of $x$}, which we denote by $\mathrm{lk}(x)$, is that subset of $\{ a, c, b_1, b_2 \}$ consisting of those points present in $X$.   First note that in the exceptional case in which $X = \{x\} \cup \mathrm{lk}(x)$, which would entail $n$ being relatively small, and $X$ consisting of at most the $5$ points illustrated, then $X$ itself will be contractible, with trivial fundamental group.  So from now on, assume that we have points in $X$ in addition to those of $\{x\} \cup \mathrm{lk}(x)$.  Furthermore, $\mathrm{lk}(x)$ must be non-empty, otherwise $X$ would be disconnected; we assume a choice of basepoint in  $\mathrm{lk}(x)$.
We divide and conquer, based on the form of this link.

\textbf{Case 1: $c \in \mathrm{lk}(x)$.}  In this case, we claim that $X$ is based-homotopy equivalent to $X - \{x\}$.
In fact we show that $X - \{x\}$ is a deformation retract of $X$.  Define a retraction $r\colon X \to X - \{x\}$ on each $y \in X$ by
$$r(y) = \begin{cases}  y & y \not= x\\ c & y = x.\end{cases}$$
Let $i \colon X - \{x\} \to X$ denote the inclusion. We have $r\circ i = \mathbf{id}\colon X - \{x\}  \to X - \{x\}$.  We claim that
$$H(y, t) =  \begin{cases}  y & t=0\\ i\circ r(y) & t=1\end{cases}$$
defines a (continuous) homotopy $H \colon X \times I_1 \to X$.   To confirm continuity, suppose we have $(y, t) \sim (y', t')$ in   $X \times I_1 $.  If neither $y$ nor $y'$ are $x$, then we have $H(y, t) = y$ and $H(y', t') = y'$, which are adjacent because $(y, t) \sim (y', t')$ implies $y \sim y'$ in $X$.  If both $y$ and $y'$ are $x$, then we have $H(x, 0) = x$ and $H(x, 1) = c$, which are adjacent. So the only remaining adjacencies we need check are for $H(x, 0)$ and  $H(y', t')$ and for $H(x, 1)$ and  $H(y', t')$ with $y' \not= x$ but $y' \sim_X x$.  But then we have $H(x, 0) = x \sim y' = H(y', t')$, and $H(x, 1) = c \sim y' = H(y', t')$.  This latter follows because the only possibilities for $y' \sim x$ are
from $\mathrm{lk}(x)$, and each of these is also adjacent to $c$.
This completes the check of the continuity of $H$, so it is a homotopy $\mathbf{id}_X \approx i\circ r\colon X \to X$. With a choice of basepoint in $\mathrm{lk}(x)$, $H$ is evidently a based homotopy $\mathbf{id}_X \approx i\circ r$.  Then, as claimed,   $X$ is based-homotopy equivalent to $X - \{x\}$ and thus these two digital images have isomorphic fundamental groups.  Since  $X - \{x\}$ has $n$ vertices, its fundamental group is a free group, by our induction hypothesis, and so this establishes the induction step in this case. 

For the remainder of the argument, we suppose that $c$ is absent, so that $\mathrm{lk}(x) \subseteq \{ a,  b_1, b_2 \}$.

\textbf{Case 2: $\{ b_1, b_2 \} \subseteq  \in \mathrm{lk}(x)$.}  In this case, we claim that $X$ is based-homotopy equivalent to $X - \{b_2\}$.
Again,  we show that $X - \{b_2\}$ is a deformation retract of $X$.  This is similar to the previous case,  Define a retraction $r\colon X \to X - \{b_2\}$ on each $y \in X$ by
$$r(y) = \begin{cases}  y & y \not= b_2\\ b_1 & y = b_2.\end{cases}$$
Let $i \colon X - \{b_2\} \to X$ denote the inclusion. We have $r\circ i = \mathbf{id}\colon X - \{b_2\}  \to X - \{b_2\}$.  We claim that
$$H(y, t) =  \begin{cases}  y & t=0\\ i\circ r(y) & t=1\end{cases}$$
defines a (continuous) homotopy $H \colon X \times I_1 \to X$.   To confirm continuity, suppose we have $(y, t) \sim (y', t')$ in   $X \times I_1 $.  If neither $y$ nor $y'$ are $b_2$, then we have $H(y, t) = y$ and $H(y', t') = y'$, which are adjacent because $(y, t) \sim (y', t')$ implies $y \sim y'$ in $X$.  If both $y$ and $y'$ are $b_2$, then we have $H(b_2, 0) = b_2$ and $H(b_2, 1) = b_1$, which are adjacent. So the only remaining adjacencies we need check are for $H(b_2, 0)$ and  $H(y', t')$ and for $H(b_2, 1)$ and  $H(y', t')$ with $y' \not= b_2$ but $y' \sim_X b_2$.  But then we have $H(b_2, 0) = b_2 \sim y' = H(y', t')$, and $H(b_2, 1) = b_1 \sim y' = H(y', t')$.  This latter follows because the only possibilities for $y \sim b_2$ are
from $\{ x, b_1, z_1, z_2 \}$ (see the figure below, and recall that there are no points in $X$ with first coordinate greater than that of $x$---also, we are supposing  $c$ is absent from $X$, but it does not affect the argument here even if we include it), and each of these is also adjacent to $b_1$.
$$
\begin{tabular}{c|c}
$a$ & \\
\hline \\
$c$ & $x$ \\
\hline \\
$b_1$ & $b_2$ \\
\hline \\
$z_1$ & $z_2$ \\
\end{tabular}
$$
This completes the check of the continuity of $H$, so it is a homotopy $\mathbf{id}_X \approx i\circ r\colon X \to X$. If we choose, say, basepoint $b_1$, then $H$ is evidently a based homotopy $\mathbf{id}_X \approx i\circ r$.  Then the induction step goes through in this case just as in the previous case.

\textbf{Case 3: $\mathrm{lk}(x) = \{ a\}$, $\mathrm{lk}(x) = \{ b_1\}$, or $\mathrm{lk}(x) = \{ b_2\}$.}  
 In this case, choose whichever point is in  $\mathrm{lk}(x)$ as basepoint, and set $U = \{x\} \cup \mathrm{lk}(x)$ and $V = X - \{x\}$.  Then $X = U \cup V$ and $U \cap V = \mathrm{lk}(x)$, a single point (the basepoint).  Furthermore, the complements $U' = X -  (\{x\} \cup \mathrm{lk}(x))$ and $V' = \{ x\}$ are disconnected.  It follows from \corref{cor: contractible U cap V}
that  we have
 $$\pi_1(X; x_0) \cong \pi_1(U; x_0) \ast \pi_1(V; x_0),$$
the free product of groups. Furthermore, $U$ is isomorphic to the unit interval $I_1$, which is (based) contractible and so we have  $\pi_1(U; x_0) \cong \{\mathbf{e}\}$.  Thus $\pi_1(X; x_0) \cong \pi_1(V; x_0)$, and our induction hypothesis gives that  $\pi_1(V;x_0)$ is a free group, as $V$ has fewer points than $X$.  The  induction step is complete in this case also.

The only remaining possibilities for the link of $x$, now are $\mathrm{lk}(x) = \{ a, b_1 \}$ and $\mathrm{lk}(x) = \{ a,  b_2 \}$.  Notice that, here, we cannot use the same sets $U$ and $V$ to decompose $X$ as in the previous case, since \thmref{thm: DVK} requires a \emph{connected} intersection $U \cap V$. Instead, we take up each of these cases with an  argument that uses the contractible path material earlier in the section.

\textbf{Case 4: $\mathrm{lk}(x) = \{ a, b_1 \}$.}  Recall that we assume $X \not= \{x\} \cup \mathrm{lk}(x)$ and so there are points in $X$ other than those adjacent to $x$, and there exists a path from $x$ to those points but only by passing through either $a$ or $b_1$.  There are two sub-cases: (W) in which $a$ and $b_1$ are not connected by a path in $X - \{x\}$; and (C) in which  $a$ and $b_1$ are connected by a path in $X - \{x\}$.  Take sub-case (W) first.  Here, $X - \{x\}$ must fall into the two components defined by
$$X_a = \big\{ p \in X - \{x\} \mid p \text{ is connected to } a \text{ by  a path in }  X - \{x\} \big\}$$
and
$$X_{b_1} = \big\{ p \in X - \{x\} \mid p \text{ is connected to } b_1 \text{ by a path in }  X - \{x\} \big\},$$
the first of which contains $a$ and the second of which contains $b_1$.  We explain this assertion as follows.  Take $y \in  X - \{x\}$ and suppose that $y \not\in X_a$.  Since $X$ is connected, there is some path from $y$ to $b_1$ in $X$.  If this path does not contain $x$, then it is a path in $X - \{x\}$ from $y$ to $b_1$ and so $y \in X_b$. Otherwise, consider the first occurrence of $x$ in the path from $y$ to $b_1$.  The part of the path from $y$ to the point preceding this first occurrence of $x$ is a path in $X - \{x\}$ from $y$ to a point of $\mathrm{lk}(x)$.  But if this point is $a$, we would have $y \in X_a$, which we said was not the case.  Therefore, it is $b_1$ and we have $y \in X_{b_1}$.  Furthermore, not only do we have $X - \{x\} = X_a \cup X_{b_1}$ but these components must be disconnected, in the sense that no point of $X_a$ is adjacent to any point of $X_{b_1}$.  For if we were to have $p \in X_a$ and $q \in X_{b_1}$ with $p \sim q$, then we could concatenate a path in  $X - \{x\}$ from $a$ to $p$ with a path in  $X - \{x\}$ from $q$ to $b_1$, to obtain a path in $X - \{x\}$ from $a$ to $b_1$.  But we are currently assuming there is no such path.  From all this, it follows that if we set $U = X_a \cup \{x\}$ and $V = X_{b_1}\cup \{x\}$, then $U \cap V = \{ x \}$.  So take the basepoint as $x_0 = x$, and notice that the complements $U' = X_{b_1}$ and $V' = X_a$ are disconnected, by the preceding discussion.  In effect, we have identified $X$ as a one-point union $U \vee V$.  It follows from \corref{cor: contractible U cap V}
that  we have
 $$\pi_1(X; x_0) \cong \pi_1(U; x_0) \ast \pi_1(V; x_0).$$
Snce $U$ and $V$ both contain fewer points than $X$, our induction hypothesis gives that their fundamental groups are free groups.  The free product of free groups is again a free group, and the induction step is complete in this sub-case (W).  

Now go back to sub-case (C), in which $a$ and $b_1$ are connected by a path in $X - \{x\}$.  By \lemref{lem: contract path} we may shorten this path to a contractible path in $X - \{x\}$.  So without loss of generality, suppose we have a contractible path $P$ in $X - \{x\}$ from $a$ to $b_1$.  Now set $U = P \cup \{x\}$ and $V = X - \{x\}$.  Then $U \cap V = P$, which is connected and contractible.  If $X \not= P \cup \{x\}$, choose $a = x_0$ and observe that the complements $U' =  X - (P \cup \{x\})$ and $V' = \{x\}$ are disconnected.  We may apply \corref{cor: contractible U cap V} again to obtain
 $$\pi_1(X; x_0) \cong \pi_1(U; x_0) \ast \pi_1(V; x_0).$$
Then both  $U$ and $V$ contain fewer points than $X$, and it follows as in the previous sub-case that $\pi_1(X; x_0)$ is a free group. However, it is possible that we have $X = P \cup \{x\}$, in which case $V$ will have (one) fewer points than $X$, but $U$ will have the same number, and we are not able to apply our inductive hypotheses to $U$.  But in this situation, notice that $U = P \cup \{x\}$  is a digital circle, with $\pi_1(U; a) \cong \Z$ by \thmref{thm: pi of C is Z}.   Now $\Z$ is a free group, and our inductive hypothesis applied to $V$ yields
$\pi_1(X; x_0) \cong  \Z \ast \pi_1(V; x_0)$: a free product of free groups, and hence a free group.  So we have closed the induction in sub-case (C).  This completes the induction in Case 4.

\textbf{Case 5: $\mathrm{lk}(x) = \{ a, b_2 \}$.}  
This case may be handled with an argument identical to that just used for Case 4, only replacing $b_1$ with $b_2$.  We omit the details.

This exhausts all cases for the link of $x$, and completes the induction.  The result follows.
\end{proof}


\begin{thebibliography}{10}

\bibitem{ADFQ03}
R.~Ayala, E.~Dom\'{i}nguez, A.~R. Franc\'{e}s, and A.~Quintero, \emph{Homotopy
  in digital spaces}, Discrete Appl. Math. \textbf{125} (2003), no.~1, 3--24,
  9th International Conference on Discrete Geometry for Computer Imagery (DGCI
  2000) (Uppsala).

\bibitem{Bo99}
L.~Boxer, \emph{A classical construction for the digital fundamental group}, J.
  Math. Imaging Vision \textbf{10} (1999), no.~1, 51--62. \MR{1692842}

\bibitem{Evako2006}
A.~V. Evako, \emph{Topological properties of closed digital spaces: One method
  of constructing digital models of closed continuous surfaces by using
  covers}, Computer Vision and Image Understanding \textbf{102} (2006),
  134--144.

\bibitem{Kong89}
T.~Y. Kong, \emph{A digital fundamental group}, Computers and Graphics
  \textbf{13} (1989), 159--166.

\bibitem{Ko-Ro91}
T.~Y. Kong and A.~Rosenfeld, \emph{Digital topology: a comparison of the
  graph-based and topological approaches}, Topology and category theory in
  computer science ({O}xford, 1989), Oxford Sci. Publ., Oxford Univ. Press, New
  York, 1991, pp.~273--289. \MR{1145779}

\bibitem{LOS19c}
G.~Lupton, J.~Oprea, and N.~Scoville, \emph{A fundamental group for digital
  images}, arXiv:1906.05976 [math.AT], 2019.

\bibitem{LOS19a}
\bysame, \emph{Homotopy theory in digital topology}, arXiv:1905.07783
  [math.AT], 2019.

\bibitem{LOS19b}
\bysame, \emph{Subdivision of maps in digital topology}, arXiv:1906.03170
  [math.AT], 2019.

\bibitem{Mas77}
William~S. Massey, \emph{Algebraic topology: an introduction}, Springer-Verlag,
  New York-Heidelberg, 1977, Reprint of the 1967 edition, Graduate Texts in
  Mathematics, Vol. 56. \MR{0448331}

\bibitem{Mas91}
\bysame, \emph{A basic course in algebraic topology}, Graduate Texts in
  Mathematics, vol. 127, Springer-Verlag, New York, 1991. \MR{1095046}

\bibitem{Mau96}
C.~R.~F. Maunder, \emph{Algebraic topology}, Dover Publications, Inc., Mineola,
  NY, 1996, Reprint of the 1980 edition. \MR{1402473}

\bibitem{Pet12}
James~F. Peters and Piotr Wasilewski, \emph{Tolerance spaces: origins,
  theoretical aspects and applications}, Inform. Sci. \textbf{195} (2012),
  211--225. \MR{2904846}

\bibitem{Po71}
T.~Poston, \emph{Fuzzy geometry}, Thesis, University of Warwick, 1971.

\bibitem{Ro86}
A.~Rosenfeld, \emph{`{C}ontinuous' functions on digital pictures}, Pattern
  Recognition Letters \textbf{4} (1986), 177--184.

\bibitem{Rot95}
Joseph~J. Rotman, \emph{An introduction to the theory of groups}, fourth ed.,
  Graduate Texts in Mathematics, vol. 148, Springer-Verlag, New York, 1995.
  \MR{1307623}

\bibitem{So86}
A.~B. Sossinsky, \emph{Tolerance space theory and some applications}, Acta
  Appl. Math. \textbf{5} (1986), no.~2, 137--167. \MR{823824}

\end{thebibliography}

\providecommand{\bysame}{\leavevmode\hbox to3em{\hrulefill}\thinspace}
\providecommand{\MR}{\relax\ifhmode\unskip\space\fi MR }
\providecommand{\MRhref}[2]{%
  \href{http://www.ams.org/mathscinet-getitem?mr=#1}{#2}
}
\providecommand{\href}[2]{#2}

\end{document}